\documentclass[11pt]{amsart}
\usepackage[mathscr]{eucal}
\usepackage{palatino, mathpazo, amsfonts, mathrsfs, amscd}
\usepackage[all]{xy}
\usepackage{url}
\usepackage{amssymb, amsmath, amsthm}
\usepackage{stackrel}
\usepackage{bm}

\usepackage{tikz-cd}
\usepackage{tikz}
\usetikzlibrary{arrows}

\newtheorem{theorem}{Theorem}[section]
\newtheorem{lemma}[theorem]{Lemma}

\newtheorem{proposition}[theorem]{Proposition}

\theoremstyle{remark}
\newtheorem{remark}[theorem]{Remark}
\newtheorem{example}[theorem]{Example}
\newtheorem{definition}[theorem]{Definition}
\newtheorem{notation}[theorem]{Notation}

\numberwithin{equation}{subsection}
\usepackage{todonotes}
\newtheorem{assumption}[theorem]{Assumption}

\makeatletter
\def\imod#1{\allowbreak\mkern10mu({\operator@font mod}\,\,#1)}
\makeatother

\newcommand{\sslash}{\mathbin{/\mkern-6mu/}}

\newcommand{\spec}{\operatorname{Spec}}

\newcommand{\cc}[1]{\mathcal{#1}}  

\newcommand{\CC}{\mathbb{C}}
\newcommand{\ZZ}{\mathbb{Z}}

\newcommand{\PP}{\mathbb{P}}
\newcommand{\QQ}{\mathbb{Q}}
\newcommand{\UU}{\mathbb{U}}

\newcommand{\sWbar}{\overline{\mathscr{W}}}
\newcommand{\sMbar}{\overline{\mathscr{M}}}

\newcommand{\bq}{\bm{q}}
\newcommand{\bg}{\mathbf{g}}

\newcommand{\bt}{\bm{t}}
\newcommand{\omlog}{\omega_{\cc C, \op{log}}}
\newcommand{\bs}{\bm{s}}

\newcommand{\ii}{\mathbb{1}}

\newcommand{\jj}{\mathfrak{j}}

\newcommand{\bv}[1]{\mathbf{#1}}

\newcommand{\dabsfact}[1]{\operatorname{D}(#1)}
\newcommand{\cX}{\cc X}
\newcommand{\cE}{\cc E}
\newcommand{\cXR}{{\cc Y_-/\CC^*_R}}
\newcommand{\cY}{\cc Y}

\newcommand{\cV}{\cc V}
\newcommand{\cYR}{{\cc Y_+/\CC^*_R}}
\newcommand{\cZ}{\cc Z}
\newcommand{\vgit}{\op{vGIT}_l}
\newcommand{\vgitz}{\op{vGIT}_0}
\newcommand{\orlov}{\op{Orlov}_l}
\newcommand{\orlovz}{\op{Orlov}_0}
\newcommand{\age}{\iota}
\newcommand{\sV}{\mathscr V}
\newcommand{\sL}{\mathscr L}

\DeclareMathOperator{\Gr}{Gr}
\DeclareMathOperator{\taut}{taut}

\DeclareMathOperator{\ch}{ch}

\DeclareMathOperator{\tot}{tot}

\newcommand{\br}[1]{\left\langle#1\right\rangle}  
\newcommand{\fl}[1]{\left\lfloor#1\right\rfloor_l}
\newcommand{\set}[1]{\left\{#1\right\}}  
\newcommand{\op}[1]{\operatorname{#1}}

\title{Integral Transforms and Quantum Correspondences}
\author[Shoemaker]{Mark Shoemaker}
\address{
  \begin{tabular}{l}
   Mark Shoemaker \\
   \hspace{.1in} Colorado State University \\
      \hspace{.1in} Department of Mathematics \\
   \hspace{.1in} 1874 Campus Delivery, Fort Collins, CO, USA, 80523-1874\\
   \hspace{.1in} Email: {\bf mark.shoemaker@colostate.edu} \\
  \end{tabular}
}

\begin{document}
\maketitle
%
%
%
%

\begin{abstract}
We reframe a collection of well-known comparison results in genus zero Gromov--Witten theory in order to relate these to integral transforms between derived categories.  This implies that various comparisons among Gromov--Witten theories and FJRW theory are compatible with the integral structure introduced by Iritani.  We conclude with a proof that a version of the LG/CY correspondence 
relating quantum $D$-modules with Orlov's equivalence is implied by a version of the crepant transformation conjecture.
\end{abstract}

\tableofcontents

\section{Introduction}

In \cite{LPS}, a collection of correspondences between Gromov--Witten theory and FJRW theory were shown to be compatible.  
Let \[w: \CC^N \to \CC\] be a quasi-homogeneous polynomial whose total degree $d$ is equal to the sum of the weights of each variable.  Let $G \leq SL_N(\CC)$ be a diagonal subgroup of automorphisms of $w$.  Pictorially, the following square, which we call the \emph{LG/CY square}, was proven in \cite{LPS} to commute:
\begin{equation}\label{e:square}
\begin{tikzcd}
\op{GWT}_0([\CC^N/G]) \ar[d, "\text{local GW/FJRW}" '] \ar[r, leftrightarrow, "\op{CTC}"] & \op{GWT}_0(\tot(\cc O_{\PP(G)}(-d))) \ar[d, "\op{QSD}"] \\
\op{FJRW}_0(w, G) \ar[r, leftrightarrow, "\op{LG/CY}"] & \op{GWT}_0(\cZ).
\end{tikzcd}
\end{equation}
In this diagram, $\op{GWT}_0([\CC^N/G])$ is shorthand for a formal structure encoding the genus zero Gromov--Witten theory of $[\CC^N/G]$  and $ \op{GWT}_0(\tot(\cc O_{\PP(G)}(-d)))$ is the same structure on a partial crepant resolution of $[\CC^N/G]$.  
On the bottom,
$\op{FJRW}_0(w, G)$ denotes the genus zero FJRW theory of the Landau--Ginzburg model given by the pair $(w, G)$, and $\op{GWT}_0(\cZ)$ is the genus zero Gromov--Witten theory of a hypersurface $\cZ$, defined as the vanishing locus of $w$ in an appropriate finite quotient of weighted projective space $\PP(G)$.  The arrows are the crepant transformation conjecture \cite{CIT, CR, CIJ}, quantum Serre duality \cite{CG}, the LG/CY correspondence \cite{ChR}, and the local GW/FJRW correspondence \cite{LPS}.

The goal of this paper is to relate each of the above correspondences to an integral transform between appropriate derived categories.  Informally, we would like to ``lift'' the LG/CY square to the derived category to obtain a cube of relations.
Some of the above correspondences are already known to be compatible with integral transforms.  In \cite{CIJ} the (equivariant) crepant transformation conjecture was shown to be compatible with a natural Fourier--Mukai transform.  A similar result has been shown for the LG/CY correspondence on the bottom  in the case where $G$ is cyclic.  
It is compatible with Orlov's equivalence between the category of matrix factorizations for $(w, G)$ and the derived category of $\cZ$, as proven in \cite{CIR}.

In this paper we show that there are derived functors corresponding to both of the vertical arrows of \eqref{e:square} as well, after restricting to subcategories of $D([\CC^N/G])$ and $D(\tot(\cc O(-d)))$ with proper support.  Consider the following maps:
\[\begin{tikzcd} BG \ar[r, bend left, "i"] & \ar[l, bend left, "\pi"] [\CC^N/G] &  & \tot(\cc O(-d)) \ar[d, "\pi"] \\
&& \cZ \ar[r, hookrightarrow, "j"] & \PP(G).
\end{tikzcd}\]
We prove the following:
\begin{theorem}[Theorem~\ref{t:MLK} and Theorem~\ref{t:QSD}]\label{t:1}
The functors
\begin{align}\label{a:1} i_*^1 \circ \pi_*&: D([\CC^N/G])_{BG} \to D([\CC^N/G], w), \text{ and} \\ j^* \circ \pi_*&: D(\tot(\cc O(-d)))_{\PP(G)} \to D(\cZ) \nonumber
\end{align}  are compatible with the local GW/FJRW correspondence and with quantum Serre duality, respectively.  In \eqref{a:1}, $D([\CC^N/G], w)$ denotes the category of matrix factorizations of $w$, and $i^1$ denotes the inclusion as a map between Landau--Ginzburg models $(BG, 0)$ and $([\CC^N/G], w)$.
\end{theorem}
The restriction to proper support is unsurprising in retrospect, and corresponds to restricting the Gromov--Witten theory of $[\CC^N/G]$  and $\tot(\cc O_{\PP(G)}(-d))$ to the \emph{narrow cohomology}, that is the subspace of cohomology generated by forms with compact support \cite[Appendix~2]{Sh1}.  This narrow theory is defined generally in \cite{Sh1} and is described below in Section~\ref{ss:spec}.  The theorem above requires a reformulation of both the local GW/FJRW correspondence and quantum Serre duality in terms of \emph{narrow quantum $D$-modules}.  We believe this is a more natural way of describing these correspondences.  It does not make mention of equivariant Gromov--Witten theory, and leads the vertical arrows of \eqref{e:square} to become \emph{isomorphisms} of quantum $D$-modules, which the previous formulations were lacking.  

Given the LG/CY square and Theorem~\ref{t:1}, it is natural to ask whether the corresponding square of derived functors commutes as well.  It turns out it does not (see Example~\ref{countex}),
however the induced maps on $K$-theory do commute.  
\begin{theorem}[Theorem~\ref{t:Ksquare}]
The functors defined in \eqref{a:1} commute with the equivalences
\begin{align}\label{a:2}
\op{vGIT}&: D([\CC^N/G])_{BG} \leftrightarrow D(\tot(\cc O(-d)))_{\PP(G)} \\
\op{Orlov}&: D([\CC^N/G], w)  \leftrightarrow D(\cZ) \nonumber
\end{align}
after passage to $K$-theory.  
\end{theorem}

\subsubsection{CTC implies LG/CY}

The primary goal of the original paper \cite{LPS}  was to prove that the LG/CY correspondence was implied by the crepant transformation conjecture.  
This was successfully carried out via the LG/CY square of \eqref{e:square}.  However 
due to the fact that the vertical arrows of \eqref{e:square} were not isomorphisms, the verification that the bottom horizontal arrow  was implied by the top required a careful technical analysis.

By rephrasing the LG/CY square in terms of narrow quantum $D$-modules, the implication becomes much more direct.  
We explain this in Section~\ref{s:CTC}. 
More precisely, we formulate a narrow version of the crepant transformation conjecture (Theorem~\ref{t:CTCnar}) which is easily implied by the equivariant crepant transformation conjecture proven in \cite{CIJ}.   As noted above, Theorem~\ref{t:1} does not use equivariant Gromov--Witten theory, and the correspondences between narrow quantum $D$-modules are in fact isomorphisms.  It is then almost immediate that the narrow crepant transformation conjecture implies the LG/CY correspondence.  Furthermore this implication is automatically compatible with the integral transforms described previously.  

Pictorially we have a cube:
\begin{equation}
\begin{tikzcd}[cramped, sep=small]
 K([\CC^N/G])_{BG} \ar[rr, leftrightarrow, "\vgit"] \ar[rd, "i^1_* \circ \pi_*"] \ar[dd]& & K(\tot(\cc O(-d)))_{\PP(G)}  \ar[rd, "j^* \circ \pi_*"] \ar[dd] &  \\
 & i^1_*\left(K(BG)\right) \ar[rr, leftrightarrow, crossing over, near start, "\orlov"] & &  j^*\left(K(\PP(G))\right) \ar[dd]\\
  QDM_{\op{nar}}(\cY_-) \ar[rr, leftrightarrow, near end, "\text{narrow CTC}"]  
  \ar[rd, leftrightarrow, "\text{local GW/FJRW}" '] & & QDM_{\op{nar}}(\cY_+)  \ar[rd, leftrightarrow, "\text{QSD}"]  &  \\
 & QDM_{\op{nar}}(w,G) \ar[from = uu, crossing over] \ar[rr, leftrightarrow, "\text{LG/CY}" '] 
  & &  QDM_{\op{amb}}(\cZ) 
\end{tikzcd}
\end{equation}
where $QDM$ denotes the quantum $D$-modules, as defined in Section~\ref{ss:spec}.  In this formulation the existence of the LG/CY arrow is immediately implied by the other three bottom arrows.  Commutativity of the front square then follows from commutativity of the other sides, and we obtain a quantum $D$-module version of the LG/CY correspondence:

\begin{theorem}[Theorem~\ref{t:LGCY}]
After pulling back via a mirror map, the narrow quantum $D$-module of $(w, G)$ is gauge equivalent to the ambient quantum $D$-module of $\cZ$.  
This identification preserves the integral structures and is compatible with Orlov's equivalence.
\end{theorem}
This theorem had previously been proven  in \cite{CIR} in the case  $G = \langle j \rangle$ a cyclic group.  

\subsection{Connections to other work}

As mentioned above, this paper builds on many previous results.  The concept of the LG/CY square and the local GW/FJRW correspondence was 
first described by the author together with Lee and Priddis in \cite{LPS}.  The formulation of the crepant transformation conjecture in terms of quantum $D$-modules and integral transforms was first described in \cite{Iri2}.  The version used in this paper was proven in \cite{CIJ}.  The LG/CY correspondence was first proven for the quintic three-fold in \cite{ChR}.  The formulation in terms of quantum $D$-modules and Orlov's equivalence was proven for $G = \langle j \rangle$ a cyclic group in \cite{CIR}.  Quantum Serre duality originated in \cite{G1} and was later generalized in \cite{CG}.  A formulation involving quantum $D$-modules and integral transforms was first given in \cite{IMM}; it is closely related to that described here.  The notion of the narrow quantum $D$-module used in this paper is described in more detail in the companion paper \cite{Sh1}, where the corresponding formulation of quantum Serre duality is proven in a general context.

It is worth remarking that each of the ``quantum'' theories appearing in this paper (Gromov--Witten theory of a GIT quotient or hypersurface, and FJRW theory of a Landau--Ginzburg model) is an examples of a
\emph{Gauged Linear Sigma Model} (GLSM), a theory whose input data consists roughly of a GIT quotient together with a \emph{potential} function. 
For simplicity, we chose to avoid 
 the language of GLSMs in this paper.  However the intuition provided by this perspective is in the background of the LG/CY square \eqref{e:square}, especially the bottom and vertical arrows.  The enumerative invariants associated to a GLSM have recently been constructed by Fan--Jarvis--Ruan in \cite{FJR15}.  An alternative construction which allows for broad insertions has been given in \cite{CFFGKS}, based on the work of \cite{PV}.  With this framework, many of the correspondences appearing in this paper and indeed the entire LG/CY square can be put into a much more general context.  A generalization of the LG/CY square to certain GLSMs was given in \cite{ClRo}.  The connection to integral transforms has yet to be proven, but should follow from similar arguments to those given in this paper.

\subsection{Acknowledgments}  I am grateful to R.~Cavalieri, E.~Clader, H.~Iritani, Y. P.~Lee, N.~Priddis, D.~Ross, Y.~Ruan and Y.~Shen for many useful conversations about the various correspondences among Gromov--Witten and FJRW  theories used in this paper.  In particular, this paper is indebted to Y. P.~Lee, who first proposed the LG/CY square described above.  I also thank N.~Addington, D.~Favero, and I.~Shipman for many patient explanations and discussions of the derived category and categories of factorizations.
This work was partially supported by NSF grant DMS-1708104.

\section{Setup}\label{s:setup}
Let $c_1, \ldots, c_N$ be positive integers and consider the corresponding action of 
$\CC^*$ on  $V \cong \CC^N$
by 
\[\alpha \cdot (x_1, \ldots, x_N) = (\alpha^{c_1} x_1, \ldots , \alpha^{c_N}x_N).\]
Let $w: V \to \CC$ be a nondegenerate polynomial, homogeneous of degree $d$ with respect to the given grading.  Explicitly, 
for $\alpha \in \CC^*$, 
\[w(\alpha^{c_1} x_1, \ldots, \alpha^{c_N} x_N) = \alpha^d w(x_1, \ldots, x_N).\]
We will assume always that $w$ is a \emph{Fermat} polynomial.  This is in order to apply the main theorems of \cite{LPS}, which are only known to hold in the Fermat case.

Let $q_j := c_j/d$ for $1 \leq j \leq N$ and define
\[j := (\exp(2 \pi i q_1), \ldots, \exp(2 \pi i q_N))\]
to be the automorphism of $V$ given by the first primitive $d$th root of unity in $\CC^*$.  Let 
\[G_{max}(w) := (\CC^*)^N \cap \op{Aut}(w)\]
denote the maximal group of diagonal symmetries of $w$.  Note that since $w$ is Fermat, $G_{max}$ will be isomorphic to 
$\ZZ_{d/c_1} \oplus \cdots \oplus \ZZ_{d/c_N}$.
We say a subgroup $G \leq G_{max}(w)$ is \emph{admissible} if $G$ contains $j$.
\begin{definition}
A \emph{(gauged) Landau--Ginzburg pair} is a pair $(w, G)$ where $w$ is a nondegenerate quasi-homogeneous polynomial and $G$ is an admissible subgroup of $G_w$.
\end{definition}

\begin{definition}
Given an element $g \in G$, the action of $g$ on $V \cong \CC^N$ is given by
\[ ( \exp( 2 \pi i m_1(g)), \ldots ,  \exp( 2 \pi i m_N(g))),\]
where $0 \leq m_j(g) < 1$.  
The ratio $ m_j(g)$ is the
\emph{multiplicity} of $g$ on the $j$th factor of $V$.
\end{definition}

\begin{definition}
An element $g \in G$ is \emph{narrow} if $(V)^g = \{0\}$.  We let $G_{\op{nar}} \subset G$ denote the subset of all narrow elements in $G$.
\end{definition}

By the Fermat assumption, $G$ splits as $\langle j \rangle \oplus \bar G$.  We will uniquely fix the splitting by requiring that $\bar G$ acts trivially on the first coordinate of $\CC^N$.
Define
\[\PP(G) := [\PP(c_1, \ldots , c_N)/\bar G]\]
where $\PP(c_1, \ldots , c_N)$ denotes weighted projective space, viewed as a stack. 
Let us fix notation for the two main orbifolds which will occupy us through the rest of this paper:
\begin{align*}
\cY_- &:= [V/G]\\
\cY_+& := \tot(\cc O_{\PP(G)}(-d)).
\end{align*}

If $\sum_{j=1}^N c_j = d$, then $\cY_+$ is a crepant (partial) resolution of the coarse space $\CC^N/G$ \cite[Lemma 6.2]{LPS}. 
The spaces $\cY_-$ and $\cY_+$ are related by variation of GIT.   To see this, let $\widetilde G =\CC^* \times \bar G$ act on $\widetilde V = V \times \CC$ as follows:  $\bar G$ acts on $V$ as before and trivially on the last factor $\CC$, and  $\CC^*$ acts with weight $c_j$ on the $j$th factor of $V = \CC^N$ and with weight $-d$ on the last factor $\CC$.  Then \begin{align}\label{e:GIT1} [\widetilde V \sslash_- \widetilde G] &= \cY_-\\ [\widetilde V \sslash_+ \widetilde G] &= \cY_+ \nonumber \end{align} where the GIT quotients are taken with respect to characters which are trivial on $\bar G$ and of weights $+/-1$ (respectively) with respect to the $\CC^*$ factor.  

We consider an additional \emph{R-charge} action on these spaces, which will be important when considering categories of factorizations.  This is given by the action of a torus $\CC^*_R = \CC^*$ on $\widetilde V$ with weights $(0, \ldots , 0, 1)$.  By \eqref{e:GIT1} this induces an action on $\cY_-$ and $\cY_+$.  
Define $\widetilde \Gamma = 
\widetilde G \times \CC^*_R$.
Consider the characters of $\widetilde \Gamma$ given by
\begin{align*}
\theta_-:  (\alpha, \bar g, \lambda_R) & \mapsto \alpha^{-d}  \lambda_R\\
\theta_+:  (\alpha, \bar g, \lambda_R) & \mapsto \alpha.
\end{align*}
One checks that 
\begin{align}\label{e:GIT2}
[\widetilde V \sslash_- \widetilde \Gamma]:= [\widetilde V \sslash_{\theta_-} \widetilde \Gamma] &= \cXR\\
[\widetilde V \sslash_+ \widetilde \Gamma] := [\widetilde V \sslash_{\theta_+} \widetilde \Gamma] & = \cYR \nonumber
\end{align}
where we take the stack quotient by $\CC^*_R$ in the last column.  For ease of notation we suppress the brackets.
Note that $\cXR$ is isomorphic to $[ V/\widetilde G]$.

Let $\tilde w: \widetilde V \to \CC$ be defined by $\tilde w = p\cdot w$ where $w$ is the original $G$-invariant potential on $\CC^N$ and $p$ is the coordinate function on the last factor of $\CC$.  Note that $\tilde w$ is $\widetilde G$-invariant and homogeneous of degree $1$ with respect to the $R$-charge action.  Thus $\tilde w$ descends to a homogeneous potential  on $[\widetilde V \sslash_- \widetilde \Gamma]$ and $[\widetilde V \sslash_+ \widetilde \Gamma]$.  By abuse of notation we will denote all these functions by $w$.

\section{Categories of factorizations}

\begin{definition}[Definition 3.1 of \cite{PV3}]
An algebraic stack $\cX$ is called a nice quotient stack if $\cX = [T/H]$ where $T$ is a noetherian scheme and $H$ is a reductive linear algebraic group such that $T$ has an ample family of $H$-equivariant line bundles.
\end{definition}

Let $\cX$ be a nice quotient stack and let $\cc L$ be a line bundle on $\cX$ with a section $s \in \Gamma(\cX, \cc L)$.  A factorization is the data $\cc E_\bullet =  (\cc E_{-1}, \cc E_0, \phi_{-1}, \phi_0)$ where $\cc E_i$ are quasi-coherent sheaves and $\phi_{-1}: \cc E_{-1} \to \cc E_0$, $\phi_0: \cc E_0 \to \cc E_{-1} \otimes \cc L$ are maps satisfying:
\begin{align*}
\phi_0 \circ \phi_{-1} &= id_{\cc E_0} \otimes s \\
\phi_{-1} \otimes id_{\cc L} \circ \phi_0 &= id_{\cc E_{-1}} \otimes s.
\end{align*}

\begin{example}
Given $\cX$ a nice quotient stack and $s$ a section of a line bundle $\cc L$ on $\cX$, let $\cc F$ be a vector bundle on $\cX$.  Suppose there exist maps $\beta \in \Gamma(\cX, \cc F)$ and $\alpha \in \hom(\cc F, \cc L)$ such that the composition $\alpha \circ \beta = s$.  Then one can define the so-called \emph{Koszul factorization} $\{\alpha, \beta\}$ as follows:
\begin{align*} \{\alpha, \beta\}_0 &= \bigoplus_{k} \wedge^{2k} \cc F^\vee \oplus_{\cc O_{\cX}} \cc L^k\\
\{\alpha, \beta\}_{-1} &= \bigoplus_{k} \wedge^{2k+1} \cc F^\vee \oplus_{\cc O_{\cX}} \cc L^k
\end{align*}
and 
\[\phi_0, \phi_{-1} := \bullet \ \lrcorner \ \beta + \bullet \wedge \alpha\]
i.e. the maps are given by all possible contractions with respect to $\beta$ and wedges with respect to $\alpha$.
\end{example}
\begin{remark}
By a mild abuse of notation, given a section $\beta \in \Gamma(\cX, \cc F)$ of a vector bundle as above, we will use $\{0, \beta\}$ to denote the Koszul complex $(\wedge^\cdot \cc F, \bullet \ \lrcorner \ \beta)$ if we are working with the derived category $D(\cX)$.  
\end{remark}

There is a shift operator $[1]$ defined by 
\[\cc E_\bullet[1] = (\cc E_0, \cc E_{-1} \otimes \cc L, -\phi_0, -\phi_1 \otimes id_{\cc L}).\]
The set of factorizations can be given the structure of a $\ZZ$-graded dg category, denoted $Fact(\cX, s)$.
Morphisms 
of degree $2k$ are given by 
\[\hom_{Fact(\cX, s)}^{2k}(\cc E_\bullet, \cc F_\bullet) := \hom_{Qcoh(\cX)}(\cc E_0, \cc F_0 \otimes \cc L^k) \oplus \hom_{Qcoh(\cX)}(\cc E_{-1}, \cc F_{-1} \otimes \cc L^k),\]
and morphisms of degree $2k +1$ are given by 
\[\hom_{Fact(\cX, s)}^{2k+1}(\cc E_\bullet, \cc F_\bullet) := \hom_{Qcoh(\cX)}(\cc E_0, \cc F_{-1} \otimes \cc L^{k+1}) \oplus \hom_{Qcoh(\cX)}(\cc E_{-1}, \cc F_0 \otimes \cc L^k).\]
There is a notion of acyclic factorizations, denoted $Acyc(\cX, s)$, which are those that arise as the totalization of an exact complex of factorizations \cite[Definition 2.3.5]{BFK}.  Denote by $[C]$ the homotopy category of a small $k$-linear dg category $C$.
\begin{definition}
The \emph{derived category of quasi-coherent factorizations}, $D^{abs}[Fact(\cX, s)]$, is the Verdier quotient of $[Fact(\cX, s)]$ by $[Acyc(\cX, s)]$.  The \emph{derived category of coherent factorizations} of $(\cX, s)$, $D(\cX, s)$, is the full subcategory of $D^{abs}[Fact(\cX, s)]$ generated by factorizations with coherent components.  We will also call $D(\cX, s)$ simply the derived category of $(\cX, s)$.
\end{definition} 
\begin{notation}
We denote by $D(\cX)$ the bounded derived category of coherent sheaves on $\cX$.
\end{notation}
\begin{remark}
Let $0$ denote the zero section of the trivial line bundle on $\cX$.  Let $\CC^*$ act on $\cX$ trivially.  The category of factorizations $D(\cX/\CC^*, 0)$ is equivalent to $D(\cX)$. 
\end{remark}
We also record the following useful comparison.  Let $X$ be a smooth variety with an action of an affine algebraic group $G$.  Let $Y$ denote the total space of a $G$-equivariant line bundle $E$ on $X$.  Let $f$ denote a regular section on $E^\vee$.  Let $w$ denote the induced potential function on $Y$.  Let $\CC^*$ act on $Y$ by scaling the fibers, and let $\eta$ denote the trivial line bundle of weight one with respect to the $\CC^*$ action.  Then $w$ can be viewed as a section of $\cc O_{[Y/\CC^*]}(\eta)$.  Let $Z$ denote the zero locus of $X$.  Consider the following diagram
\begin{equation}\label{e:csquare}\begin{tikzcd}
Y|_Z \ar[r, "j' "] \ar[d, "\pi|_Z"] & Y \ar[d, "\pi"]\\
Z \ar[u, bend left, " i '"] \ar[r, "j"]  & X
  \end{tikzcd}
  \end{equation}
  
  The following result was proven independent by Isik \cite{Isik} and Shipman \cite{Shi}, where it appears as the main theorem in each.  The generalization to the setting of quotients by $G$ is due to Hirano, \cite[Proposition~4.8]{Hirano}.
  \begin{theorem}\cite{Isik, Shi, Hirano}\label{t:ish}
  There is an equivalence of categories given by
  \[\phi_+ = j'_* \circ \pi|_Z^*: D(Z/G) = D(Z/(G \times \CC^*), 0) \to D(Y/(G \times \CC^*), w).\]

  \end{theorem}

 \begin{example}\label{e:S1} Consider the vector bundle $\pi^*(E)$ over $[Y/G]$.  This bundle has a tautological section $\taut \in \Gamma(Y, \pi^*(E))$.  
 The pullback of $f$ yields a map $\pi^*(f) \in \hom(\pi^*(E), \cc O(\eta))$.  
 Note that the composition $\pi^*(f)\circ \taut$ is equal to $w$.  We denote by $S_1$ the corresponding Koszul factorization 
 \[S_1 := \{ \pi^*(f), \taut \}.\]
 \end{example}
 \begin{definition}\label{d:p}
 Define the equivalence
 \[\tilde \phi_+ := \op{det}(E^\vee) \otimes \phi_+ [-\op{rank}(E)].\]
 \end{definition}
\begin{remark}\label{r:Ship} We note that $\tilde \phi_+$ identifies $S_1$ with $\cc O_{[Z/G]}$   (Remark~2.5.6 of \cite{CFFGKS}).
\end{remark}
\subsection{Functors}
Let $f: \cX \to \cY$ be a morphism of algebraic stacks and let $s$ be a section of a line bundle $\cc L$ on $\cY$.  Then there is a pullback functor
\[ f^*: Fact(\cY, s) \to Fact(\cX, f^*(s))\]
defined simply by pulling back all of the data of a factorization on $\cY$.  There is also a pushforward 
\[ f_*: Fact(\cX, f^*(s)) \to Fact(\cY, s).\]
Finally, given a section $r$ of a $\cc L$, 
there is a tensor product
\[\otimes_{\cc O_{\cY}}: Fact (\cY, s) \otimes_k Fact(\cY, r) \to Fact(\cY, s + r). 
\]
The above operations induce derived functors $\mathbb{L} f^*$, $\mathbb{R} f_*$ (if $f: \cX \to \cY$ is proper) and $\otimes_{\cY}^\mathbb{L}$ respectively.  
Without further comment we will always work with derived functors, by abuse of notation we will just denote these by $f^*$, $f_*$ and $\otimes_{\cY}$.

The following analogue of the projection formula holds in the setting of factorizations.
\begin{proposition}\cite[Proposition~2.2.10]{CFFGKS} \label{pprojform}
Let $f: \cX \to \cY$ be a proper morphism, let $s$ and $r$ be sections of a line bundle on $\cY$.  Then if $\cc E_\bullet$ is a factorization of $f^*(s)$ on $\cX$ and $\cc F$ is a factorization of $r$ on $\cY$, then 
\[ f_*( \cc E_\bullet \otimes_{\cX} f^*(\cc F)) = f_*(\cc E_\bullet) \otimes_{\cY} \cc F \in D(\cY, s + r).\]

\end{proposition}

\subsection{Chern maps}\label{s:chch}

\subsubsection{Orbifold Chern character}
Given $\cX$ a smooth Deligne--Mumford stack with quasi-projective coarse moduli space.  Let $I\cX$ denote the inertia stack of $\cX$, defined as the fiber product of the diagonal with itself.  There is a natural open and closed subset of $I\cX$ which is isomorphic to $\cX$ itself.  The other components are referred to as \emph{twisted sectors}.

Given a vector bundle $E$ on an smooth DM stack $\cX$, let $B$ index the twisted sectors of $I\cX$.  Let $E_b$ denote the restriction of $E$ to the twisted sector ${\cX}_b$.  Consider the action of the stabilizer $g_b$ on $E_b$.  This action splits $E_b$ into eigenbundles.
Let $E_{b,f}$ denote the eigenbundle of weight $f$ for $0 \leq f <1$.  In other words $g_b$ acts on $E_{b,f}$ by multiplication by $e^{2 \pi i f}$.

\begin{definition}\label{d:orbch}
Define the orbifold Chern character of $E$, $\ch(E) \in H^*_{\op{CR}}(\cX)$, by
\[ \ch(E) := \bigoplus_{b \in B} \sum_{ 0 \leq f < 1} e^{2 \pi i f} \ch(E_{b,f}),\]
where $\ch(E_{b,f})$ lies in $H^*({\cX}_b)$.  This extends to a map $\ch: D(\cX) \to H^*_{\op{CR}}(\cX)$.
\end{definition}

\subsubsection{Chern map for factorizations}
Let $C$ be a $k$-linear small dg-category.  Let $C(k)$ denote the dg category of unbounded complexes over $k$.  Consider the $C-C$ bi-module
\begin{align*}
\Delta_C: C^{op} \otimes C &\to C(k) \\
(c, c') &\mapsto \hom_C(c, c').\end{align*}
This induces the trace functor
\begin{align*}
D(C \times C^{op}) &\to D(k)\\
F & \mapsto F\otimes_{C^{op} \otimes C} \Delta_C.
\end{align*}

\begin{definition}
Let $C$ be a small $k$-linear dg-category.  We let $HH_*(C)$ denote the Hochschild homology of $C$:
\[HH_*(C) := H^{-*}(\Delta_C \otimes_{C^{op} \otimes C} \Delta_C).\]
If $C = D(\cY, w)$, 
we denote 
$HH_*(C)$ simply by $HH_*(\cY, w)$.
\end{definition}
In \cite{BFK2} and \cite{PV2}, the Hochschild homology of $D(\cXR, w)$ is calculated, as well as the Chern map.
\begin{proposition}[Theorems~2.5.4 and~3.3.3 of \cite{PV2} and Theorem 1.2 of \cite{BFK2}]

The Hochschild homology $HH_*(\cXR, w)$ is isomorphic to
\[\bigoplus_{g \in G} (\cc Q_{w|_{(\CC^N)^g}})^g\]
where $\cc Q_{w|_{(\CC^N)^g}}$ denotes the Jacobian ring of the restriction of $w$ to $(\CC^N)^g$.
 Given a factorization $ \cc E_\bullet = (\cc E_{-1}, \cc E_0, \phi_{-1}, \phi_0)$, then under this identification
\[\ch(\cc E_{-1}, \cc E_0, \phi_{-1}, \phi_0) := \bigoplus_{g \in G} 
\left[ \text{str}(\partial_{j_1} \phi_{\cc E} \circ \partial_{j_2} \phi_{\cc E} \circ \cdots \circ \partial_{j_{N_g}} \phi_{\cc E} \circ
g)|_{(\CC^N)^g} dx_{j_1} \wedge \cdots \wedge dx_{j_{N_g}} \right] \]
where $\cc E = \cc E_0 \oplus \cc E_{-1}$, $\phi_{\cc E} = \phi_0 \oplus \phi_{-1}$ is block anti-diagonal, and $g$ denotes the action of the isotropy generator on $\cc E$.
\end{proposition}

Note that if an element $g \in G$ fixes only the origin, then
the corresponding summand of $HH_*(\cXR, w)$ is isomorphic to $\CC$.  We let $\varphi_{gj^{-1}}$ denote the unit generator of this summand.  Thus 
\[\sum_{g \in G_{\op{nar}}} \CC \cdot \varphi_{gj^{-1}} \subset HH_*(\cXR, w).\]
This subspace is termed the \emph{narrow subspace} and will feature in what follows.

\subsection{Koszul factorizations and explicit computations}


\subsubsection{The affine phase $\cY_-$}
Let $\eta: \widetilde \Gamma \to \CC^*$ denote the character sending $(\alpha, \bar g, \lambda_R)$ to $\lambda_R$.  
This defines a line bundle $\cc O_{\cXR}(\eta)$ on $\cXR$ by \eqref{e:GIT2}. Let $0_\eta $ denote the zero section of $\cc O_{\cXR}(\eta)$.  Note that $w$ also gives a section of $\cc O_{\cXR}(\eta)$.  

Via the inclusion $\widetilde G \to \widetilde \Gamma$ given by $(\alpha, \bar g) \mapsto (\alpha, \bar g, \alpha^d)$, the character $\eta$ restricts to a character of $\widetilde G$.  This defines a line bundle on $B \widetilde G$.
By abuse of notation we will also denote this character and the corresponding line bundle by $\eta$ when no confusion will result.  
Alternatively this line bundle can be viewed as the  restriction of $\cc O_{\cXR}(\eta)$ to the origin, after applying the isomorphism between $\cXR$ and  $[ V/\widetilde G]$. 

By Proposition 1.2.2 of \cite{PV},  $D(B \widetilde G, 0_\eta) \sim D(BG)$, where we use the symbol $\sim$ to mean an equivalence of categories.  We will implicitly identify these categories in what follows.  Let $D(\cY_-)_{BG}$ denote the full subcategory of $D(\cY_-)$ consisting of complexes supported on $BG$.

Consider  the projection map $\pi: \cY_- \to BG$ and  the inclusion $i^0: BG \to \cY_-$.
These induce pushforward functors
\begin{align*} 
\pi_*&: D(\cY_-)_{BG} \to D(BG) \sim D(B \widetilde G, 0_\eta),\\
i^0_*&:   D(B \widetilde G, 0_\eta) \sim D(BG) \to D(\cY_-)_{BG}.
\end{align*}
There also exists a map of LG spaces
$ i^1: (B \widetilde G, 0_\eta) \to (\cXR,  w)$
inducing the pushforward 
\[i^1_*: D(B \widetilde G, 0_\eta) \to D(\cXR,  w).\]

The following holds by, e.g., Lemma 4.6 and 4.8 of \cite{BFK2}.
\begin{proposition}\cite[Lemmas~4.6 and~4.8]{BFK2}\label{p:strgenBG}
The category $D(\cY_-)_{BG} $ is strongly generated by $i^0_*(D(B \widetilde G, 0_\eta))$.
\end{proposition}

In what follows, we will make use of a particular factorization in calculating maps on cohomology and Hochschild homology. 

Let $V$ be the vector bundle  of Section~\ref{s:setup}.
The  actions of $G$ and $\widetilde \Gamma$ allow us to view $V$ as defining vector bundles on $BG$  and $B\widetilde \Gamma$.  To ease notation we will also use $V$ to denote the vector bundle on $\cY_-$ and $\cXR$ obtained by pulling back the vector bundles from $BG$  and $B\widetilde \Gamma$.

Over $\cY_-$ and $\cXR$, the bundle $V$ has a tautological section $\taut$, and a cosection $dw$ defined by 
\begin{equation}\label{e:cos}dw: (v_1, \ldots, v_N) \mapsto \sum_{j=1}^N v_j q_j \partial_j w.\end{equation}
Define $\beta $ to be either the tautological section $ \taut \in  \Gamma( \cXR, V)$ or $\taut \in \Gamma(\cY_-, V)$ depending on the context.
Define $\alpha := dw$ as an element of
$\op{Hom}_{\cXR}(V, \cc O_{\cXR}(\eta))$.  We consider the Koszul complex $\{0, \beta\} \in D(\cY_-)_{BG} \subset D(\cY_-)$ and the Koszul factorization $\{\alpha, \beta\} \in D(\cXR, w)$.

The cohomology $H^*_{\op{CR}}(\cY_-)$ is given by 
\[H^*_{\op{CR}}(\cY_-) = \bigoplus_{g \in G} \CC \cdot \ii_g\]
where $\ii_g$ denotes the fundamental class on the $g$th twisted sector.  An explicit calculation yields
\begin{equation}\label{e:chzerobeta}\ch\left(\{0, \beta\} \right) = \bigoplus_{g \in G} \prod_{j=1}^N \left( 1 - e^{2 \pi i (-m_j(g))} \right) \ii_g.\end{equation}  To derive \eqref{e:chzerobeta}, first expand the product $\prod_{j=1}^N \left( 1 - x e^{2 \pi i (-m_j(g))} \right)$.  The coefficient of $(-x)^k$ is seen to equal $\ch\left( \wedge^k V^\vee \right)|_{H^*({\cY_-}_g)}$.  
\begin{definition}
Define the \emph{narrow} subspace \[H^*_{\op{CR}, \op{nar}}(\cY_-) := \bigoplus_{g \in G_{\op{nar}}} \CC \cdot \ii_g \subset H^*_{\op{CR}}(\cY_-).\]
\end{definition}
Observe that if $N_g>0$ then $m_j(g) = 0$ for some $j$ and the $\ii_g$ coefficient of $\ch(\{0 ,\beta\})$ is zero.  Thus $\ch(\{0, \beta\})$ lies in $H^*_{\op{CR}, \op{nar}}(\cY_-)$.

A very similar calculation holds for the factorization $\{\alpha, \beta\}$.
By Proposition 4.3.4 of~\cite{PV2}, 
\begin{equation}\label{e:chalphabeta}\ch\left(\{\alpha, \beta\} \right) = \bigoplus_{g \in G_{\op{nar}}} \prod_{j=1}^N \left( 1 - e^{2 \pi i (-m_j(g))} \right) \phi_{gj^{-1}}.\end{equation}


\subsubsection{The geometric phase $\cY_+$}\label{s:gp}

The character $\eta: \widetilde \Gamma \to \CC^*$ also defines a line bundle $\cc O_{\cYR}(\eta)$ 
on $\cYR$.  Denote by $\cc O_{\PP(G)/\CC^*_R}(\eta)$ the restriction of this line bundle to $\PP(G)/\CC^*_R$ and by $0_\eta \in 
\Gamma(\tot(\cc O_{\PP(G)})/\CC^*, \cc O_{\PP(G)/\CC^*_R}(\eta))$ the zero section of this line bundle.  
Again by Proposition 1.2.2 of \cite{PV}, $D(\PP(G)/\CC^*_R, 0_\eta) \sim D(\PP(G))$.  As above we have the following functors
\begin{align*} 
\pi_*&:  D(\cY_+)_{\PP(G)} \to D(\PP(G)/\CC^*_R, 0_\eta) \sim D(\PP(G)) \\
i^0_*&: D(\PP(G)/\CC^*_R, 0_\eta) \sim D(\PP(G)) \to D(\cY_+)_{\PP(G)} \\
i^1_*&: D(\PP(G)/\CC^*_R, 0_\eta) \to D(\cYR,  w).
\end{align*}
As with Proposition~\ref{p:strgenBG}, again by \cite{BFK2} we have the following.
\begin{proposition}\cite[Lemmas~4.6 and~4.8]{BFK2}\label{p:strgen2}
The category $D(\cY_+)_{\PP(G)}$ is strongly generated by $D(\PP(G)/\CC^*_R, 0_\eta)$.
\end{proposition}

As in Example~\ref{e:S1}, consider the pullback of $\cc O_{\PP(G)}(-d)$ to the total space $\cY_+$.  Denote this line bundle by $\cc O_{\cY_+}(-d)$.  We will let $\cc O_{\cYR}(-d, \eta)$ denote the corresponding line bundle on $\cYR$, where the action of $\CC^*_R$ on $\cc O_{\cYR}(-d, \eta)$ is by scaling.  In either case this line bundle has a tautological section $\taut$, as well as a cosection
\[\pi^*(w) \in \hom_{\cYR}(\cc O_{\cYR}(-d, \eta), \cc \cc O_{\cYR}(\eta)).\]
As above we will be interested in the Koszul complex $\{0 , \taut\} \in D(\cY_+)$ and $\{ \pi^*(w), \taut\} \in D(\cYR,  w)$.
Note that by Remark~\ref{r:Ship}, the equivalence $\tilde \phi_+: D(\cZ) \xrightarrow{\sim} D(\cYR,  w)$ identifies $\cc O_{\cZ}$ with $\{ \pi^*(w), \taut\}$.


\subsection{The $K$-theoretic commuting square}\label{ss:Dequiv}

In this section we assume the \emph{quasi-Calabi--Yau condition}, that $\sum_{j=1}^N c_j = d$.  This ensures that $\cY_+$ and $\cY_-$ are Calabi--Yau and, in particular, are $K$-equivalent.  Let $j: \cZ \to \PP(G)$ denote the zero locus of $w$, here viewed as a section of $\cc O_{\PP(G)}(d)$.  Recall $\cY_+$ and $\cY_-$ are related by variation of GIT with respect to the action $\widetilde G$ on $V$.  
Due to the (quasi-) Calabi--Yau assumption, the derived categories $D(  \cY_+)$ and $D(\cY_-)$ are equivalent, via a functor $\vgit$ which depends on a parameter $l \in \ZZ$.  We recall the functor $\vgit$ below.  See \cite{ BFK, Seg} for details.  We follow the exposition of \cite{BFK}.

Recall $\widetilde G = \CC^* \times \bar G $.  Let $\Lambda$ denote the factor of $\CC^*$.  Let $W_{\Lambda, l}([\widetilde V/\widetilde G])$ denote the full triangulated subcategory of $D([\widetilde V/\widetilde G])$ generated by complexes of line bundles whose $\Lambda$-weight lies in the range $[l, l+d-1]$.  
Let 
\begin{align*} i_-&: [\widetilde V \sslash_- \widetilde G] \to [\widetilde V/\widetilde G] \\ i_+&: [\widetilde V \sslash_+ \widetilde G] \to [\widetilde V/\widetilde G]
\end{align*}
denote the respective inclusions.  Then it is proven in \cite{BFK} that both $i_-^*|_{W_{\Lambda, l}([\widetilde V/\widetilde G])}$ and $i_+^*|_{W_{\Lambda, l}([\widetilde V/\widetilde G])}$ are equivalences of categories.  Define
\begin{align}\label{e:evg1}
\vgit :  D(\cY_-) &\xrightarrow{\sim} D(  \cY_+) \\
 \cc E_\bullet & \mapsto i_+^* \circ (i_-^*|_{W_{\Lambda, l}([\widetilde V/\widetilde G])})^{-1}(\cc E) \nonumber
\end{align}

The subgroup $\Lambda$ may also be seen as a subgroup of $\widetilde \Gamma$ via the  inclusion $\widetilde G  = \widetilde G \times \{id\} \hookrightarrow \widetilde \Gamma =  \widetilde G \times \CC^*_R$.
Define $W_{\Lambda, l}([\widetilde V / \widetilde \Gamma], \tilde w )$ to be the full triangulated subcategory of $D([\widetilde V / \widetilde \Gamma], \tilde w )$ generated by factorizations $\cc E_\bullet =  (\cc E_{-1}, \cc E_0, \phi_{-1}, \phi_0)$ such that both $\cc E_{-1}$ and $\cc E_0$ are direct sums of line bundles of $\Lambda$-weight in the range $[l, l+d-1]$.
There is an analogous equivalence as \eqref{e:evg1} for factorization categories (see \cite{BFK}), given by 
\begin{align}
\vgit :  D(\cXR, w) &\xrightarrow{\sim} D(\cYR, w) \\
 \cc E_\bullet & \mapsto i_+^* \circ (i_-^*|_{W_{\Lambda, l}([\widetilde V / \widetilde \Gamma], \tilde w )})^{-1}(\cc E_\bullet) \nonumber
\end{align}

Recall by Corollary~5.2 of \cite{Shi} and Section~7 of \cite{BFK} that Orlov's equivalence $D(\cXR,  w) \xrightarrow{\sim} D(\cZ)$ can be viewed as a composition 
\[D(\cXR,  w) \xrightarrow{\vgit} D(\cYR,  w) \xrightarrow{\phi_+^{-1}} D(\cZ)\] 
 where $\phi_+$ was defined in Theorem~\ref{t:ish}.
For simplicity, we will take this as the definition of Orlov's equivalence, but using $\tilde \phi_+^{-1}$ from Definition~\ref{d:p} in place of $\phi_+^{-1}$:
\begin{definition} Define the categorical equivalence $\orlov$ as the composition:
\[\orlov := \vgit \circ \tilde \phi_+^{-1}.\]
\end{definition}
\begin{remark}\label{r:shifty}
One checks that this agrees with the definition in Section 4.2 of \cite{CIR} up to the shift  by $[1]$.
\end{remark}
\begin{lemma}\label{l:phij}
The functors $ \tilde \phi_+ \circ j^*, i^1_*: D(\PP(G))  \to D(\cYR,  w)$ are equal.
\end{lemma}
\begin{proof}
By Remark~\ref{r:Ship}, $\tilde \phi_+ \circ j^* (\cc O_{\PP(G)}) = \{ \pi^*(w), \taut\}$.  By the proof of Lemma 3.2 in \cite{Shi}, this is $i^1_*(\cc O_{\PP(G)})$.
Referring to diagram \eqref{e:csquare}, following the definitions and applying the projection formula of Proposition~\ref{pprojform}, we then obtain
\begin{align*}
\tilde \phi_+ \circ j^*( -) & =  \det(\cc O_{\PP(G)/\CC^*_R}(-d, \eta))[-1] \otimes j '_* ( \pi|_{\cZ}^* \circ j^*( -))\\
& = \det(\cc O_{\PP(G)/\CC^*_R}(d, \eta))[-1] \otimes j '_*( \cc O_{\cYR|_{\cZ}} \otimes {j'}^*  \circ \pi^*( -)) \\
& =  \det(\cc O_{\PP(G)/\CC^*_R}(d, \eta))[-1] \otimes j '_*( \cc O_{\cYR|_{\cZ}} )\otimes  \pi^*( -) \\
& =   \det(\cc O_{\PP(G)/\CC^*_R}(d, \eta))[-1] \otimes j '_*( \pi|_{\cZ}^* (\cc O_{\cZ} ))\otimes  \pi^*( -) \\
& = \tilde \phi_+(\cc O_{\cZ}) \otimes \pi^*(-) \\
& = i^1_*(\cc O_{\PP(G)}) \otimes \pi^*( - ) \\
&= i^1_*(\cc O_{\PP(G)} \otimes (i^0)^* \circ \pi^*( -)) \\
& =  i^1_*( - ). 
\end{align*}
\end{proof}
Let us investigate the functors $\vgit$ in more detail.
Given integers $k_1, k_2 \in \ZZ$ and a character $ \zeta$ of $\bar G$,
define a character of $\widetilde \Gamma$ by
\[ (\alpha, \bar g, \lambda_R) \xrightarrow{(k_1, \zeta, k_2)} \alpha^{k_1} \zeta(\bar g) \lambda_R^{k_2},\] and a character of $\widetilde G$ by
\[ (\alpha, \bar g) \xrightarrow{(k_1, \zeta)} \alpha^{k_1} \zeta(\bar g).\]
 Let $\cc O(k_1, \zeta, k_2)$ (resp. $\cc O(k_1, \zeta)$) denote the corresponding line bundle on 
 a given quotient of (an invariant subvariety of) $\widetilde V$ by $\widetilde \Gamma$ (resp. $\widetilde G$).
 Via the inclusion $\widetilde G  = \widetilde G \times \{id\} \hookrightarrow \widetilde \Gamma =  \widetilde G \times \CC^*_R$,  the line bundle $\cc O(k_1, \zeta, k_2)$ pulls back to $\cc O(k_1, \zeta)$. By \eqref{e:GIT1} and \eqref{e:GIT2}, the character $(k_1, \zeta, k_2): \widetilde \Gamma \to \CC^*$  (resp. $(k_1, \zeta): \widetilde G \to \CC^*$)  defines line bundles on $\cXR$, $\cYR$, $B\widetilde G$ and $\PP(G)/\CC^*_R$ (resp. $\cY_-$, $\cY_+$, $BG$ and $\PP(G)$).

Given a choice of $l \in \ZZ$, we will define a \emph{round-down} operation for factorizations $\cc E_\bullet$ on 
$[\widetilde V / \widetilde \Gamma]$ provided that the components $\cc E_0$ and $\cc E_1$ are locally free.  
First, 
given a line bundle $\cc O(k_1, \zeta, k_2) \in [\widetilde V / \widetilde \Gamma]$, there exists a unique integer $m$ such that $k_1 - md \in [l, l+d-1]$.
Then define 
\begin{equation}\label{rd1} \fl{\cc O(k_1, \zeta, k_2)} := \cc O(k_1 - md, \zeta, k_2+m).\end{equation}
By construction, $\fl{\cc O(k_1, \zeta, k_2)}$ lies in the $\Lambda$-weight range corresponding to the window $W_{\Lambda, l}([\widetilde V / \widetilde \Gamma], \tilde w )$.  Given a $\widetilde \Gamma$-equivariant locally free sheaf $\cc E$ on $V$, since $\widetilde \Gamma$ is abelian $\cc E$ will split as a direct sum of line bundles.  For $\cc E$ a locally free sheaf, define $\fl{ \cc E}$ by taking the round-down of each summand.  Next, suppose we have an $\cc O_{[\widetilde V/\widetilde \Gamma}$-linear map 
\[ \phi: \cc O(k_1, \zeta, k_2) \to \cc O(k_1', \zeta ', k_2').\] we define 
\[\fl{\phi}: \fl{\cc O(k_1, \zeta, k_2)}  
\to \fl{\cc O(k_1', \zeta ', k_2')} 
\] 
to be
\begin{align}\label{e:flmap} \fl{\phi} := \phi \cdot p^{m' - m}.\end{align} 
Note that with respect to the $\CC^*$-factor in $\widetilde G$, $\phi$ is homogeneous of degree $k_1 ' - k_1$.
By definition of $m$ and $m '$,  
\[k_1 ' - k_1 \leq (m' - m)d + (d-1) < (m' - m + 1)d.\]  The only homogeneous coordinate with negative weight with respect to the first factor of $\CC^*$ in $\widetilde \Gamma$ is $p$, 
the last homogeneous coordinate of $[\widetilde V / \widetilde \Gamma]$.
The coordinate $p$ has a weight of $-d$, so this inequality guarantees that whenever $m' < m$, the $\cc O_{[\widetilde V/\widetilde \Gamma}$-linear map $\phi$ is divisible by $p^{m - m'}$.
In particular, \eqref{e:flmap} is well-defined.   The expressions~\eqref{rd1} and~\eqref{e:flmap} define a functor
\[ \fl{\bullet}: D\left( [\widetilde V/\widetilde \Gamma]\right) \to D\left( [\widetilde V/\widetilde \Gamma]\right).\]

Given an $\cc O_{[\widetilde V/\widetilde \Gamma}$-linear map $ \phi: \cc E \to \cc F$ between locally free sheaves, after decomposing $\cc E$ and $\cc F$ as direct sums of line bundles, $\{\cc E_j\}$ and $\{\cc F_i\}$, $\phi$ can be represented as a matrix $(\phi_{ij})$ of $\widetilde \Gamma$-homogeneous functions. define $\fl{\phi}: \fl{\cc E} \to \fl{\cc F}$ to be $(\fl{\phi_{ij}})$.  Finally, we may define the round-down factorization as follows.

\begin{definition}
Given $\cc E_\bullet$ a factorization of $w \in \Gamma([\widetilde V / \widetilde \Gamma], \cc O(\eta))$ by locally free sheaves, 
define the factorization
\[\fl{\cc E_\bullet} :=  (\fl{\cc E_{-1}},\fl{ \cc E_0}, \fl{\phi_{-1}}, \fl{\phi_0}).\]
It is apparent from the construction that if $\phi_0 \circ \phi_{-1} = id_{\cc E_0} \otimes w$, then $\fl{\phi_0} \circ \fl{\phi_{-1}} = \fl{id_{\cc E_0}  \otimes w} = id_{\fl{\cc E_0}} \otimes w$ and similarly for the composition $\phi_{-1} \otimes id_{\cc L} \circ \phi_0$.  This shows that $\fl{\cc E_\bullet} $ is also a factorization of $w$.

Given a complex $\cc E_\bullet = (\cc E_i, \phi_i : \cc E_i \to \cc E_{i+1})$ of locally free sheaves   on $[\widetilde V/ \widetilde G]$, define
\[\fl{\cc E_\bullet} :=  (\fl{\cc E_i}, \fl{\phi_i}).\]
It is clear that $\fl{\cc E_\bullet}$ will also be a complex.
\end{definition}

Given a locally free factorization or complex $\cc E_\bullet$ which splits as a sum of line bundles $\cc O(k_1, \zeta, k_2)$ 
all satisfying $k_1 \geq l$,
 there exists a map \[rd_l: \cc E_\bullet \to \fl{\cc E_\bullet}\] which, on a given summand $\cc O(k_1, \zeta, k_2)$, is simply the map
\begin{align*}
p^m: \cc O(k_1, \zeta, k_2) &\to \fl{\cc O(k_1, \zeta, k_2)}  = \cc O(k_1 - md, \zeta, k_2+m).
\end{align*}   Note that $rd_l$ is injective.  
\begin{definition} In the situation above,  define $\br{ \cc E_\bullet}$ to be the cokernel of $rd_l$.  \end{definition}
It is immediate that $\br{\cc E_\bullet}$ is set-theoretically supported on $\{p = 0\}$.

On the stack $[\widetilde V / \widetilde \Gamma]$, the vector bundle $ V \subset \widetilde V$ has a natural tautological section $\taut$ and a cosection $dw$ defined as in \eqref{e:cos}.  
Consider the factorization $\{dw, \taut\}$.

Note that $\{0, \taut\}$ is simply the Koszul complex resolving $\cc O_{\{x_i = 0\}_i}$, and the cone $\text{C}( \{0, \taut\} \to  \cc O_{\{x_i = 0\}_i})$ is acyclic.
By (the proof of) Lemma 3.2 of \cite{Shi}, this implies that $\text{C}( \{dw, \taut\} \to \cc O_{\{x_i = 0\}_i})$ is zero in $D([\widetilde V / \widetilde \Gamma], \tilde w)$,  so 
\[\{dw, \taut\} \sim \cc O_{\{x_i = 0\}_i}.\]
In particular, 
\begin{align} \label{e:vgitex1}
i_-^*(\{0, \taut\}) &= i^0_*( \cc O_{BG}), \\
i_-^*(\{dw, \taut\}) & = i^1_*( \cc O_{BG}), \nonumber\\
i_+^*(\{0, \taut\}) & = 0, \nonumber\\
i_+^*(\{dw, \taut\}) & = 0,\nonumber
\end{align}
due to the fact that in $\cY_-$ the locus $\{x_i = 0\}_i$ is equal to  $BG$, whereas in $\cY_+$ the locus $\{x_i = 0\}_i$ is empty.

For $0 \leq i \leq d-1$, and $\zeta \in \widehat{\bar G}$, consider the pullback of the map of factorizations
\[rd_l:  \{0, \taut\} \otimes \cc O(l+d+i, \zeta) \to \fl{\{0, \taut\} \otimes \cc O(l+d+i, \zeta)}\]
to $\PP(G)$.  Define, in $D(\PP(G))$, the complex
\begin{align*} &K(l+d+i, \zeta) := \\
&\text{coker}\left( \{0, \taut\} \otimes \cc O_{\PP(G)}(l+d+i, \zeta) \to \fl{\{0, \taut\} \otimes \cc O_{\PP(G)}(l+d+i, \zeta)}\right).\end{align*}
  This complex can be described more explicitly.  On summands which were modified by taking the floor, the map $rd_l$ is multiplication by $p$ and so restricts to the zero map on $\PP(G)$.  On summands which were not modified, $rd_l$ restricts to the identity.  Thus $K(l+d+i, \zeta)$ is obtained from $\fl{\{0, \taut\} \otimes \cc O_{\PP(G)}(l+d+i, \zeta)}$ by simply removing all summands which were unchanged by taking the floor.  Note that the map $dw$ is zero on $\PP(G)$, thus the pullback of $\{dw, \taut\}$ is equal to the pullback of $\{0 ,\taut\}$ under the equivalence $D(\PP(G)/\CC^*_R, 0_\eta) \sim D(\PP(G))$.
We see that $K(l+d+i, \zeta)$ corresponds to
\begin{align*} \text{coker}&\left( \{dw, \taut\} \otimes \cc O_{\PP(G)}(l+d+i, \zeta, 0) \right. \to \\ &\left. \fl{\{dw, \taut\} \otimes \cc O_{\PP(G)}(l+d+i, \zeta, 0)}\right)
\end{align*}
under this equivalence.
Finally, one observes that
\begin{align} \label{e:Kfrac}
i_+^*\left(\br{\{0, \taut\} \otimes \cc O_{[\widetilde V/ \widetilde G]}(l+d+i, \zeta) }_l\right) &= i^0_*(K(l+d+i, \zeta)) \\
i_+^*\left(\br{\{dw, \taut\} \otimes \cc O_{[\widetilde V/ \widetilde G]}(l+d+i, \zeta, 0) }_l\right) &= i^1_*(K(l+d+i, \zeta)).\nonumber
\end{align}

\begin{proposition}\label{p:vgit}
For any choice of $\zeta \in \widehat{\bar G}$ and $0 \leq i \leq d-1$, the equivalences \begin{align*}\vgit:& D(\cY_-) \xrightarrow{\sim} D(  \cY_+), \\ \vgit:& D(\cXR,  w) \xrightarrow{\sim} D(\cYR,  w)\end{align*} take the following values 
\begin{align}
 i^0_*( \cc O_{BG}(l+d+i, \zeta)) &\mapsto 
i^0_*(K(l+d+i, \zeta))
\label{e:vgitex2} \\
i^1_*( \cc O_{BG}(l+d+i, \zeta)) &\mapsto
i^1_*(K(l+d+i, \zeta)). 
\label{e:vgitex3}
\end{align}
\end{proposition}

\begin{proof}
To prove \eqref{e:vgitex2}, consider
$\fl{ \{0, \taut\} \otimes \cc O_{[\widetilde V/ \widetilde G]}(k ,\zeta)}$ in $D([\widetilde V, \widetilde G])$.  Since $\fl{ \{0, \taut\} \otimes \cc O_{[\widetilde V/ \widetilde G]}(k ,\zeta)}$ lies in the window $W_{\Lambda, l}([\widetilde V/\widetilde G])$, the equivalence $\vgit$ maps $i_-^*\left(\fl{ \{0, \taut\} \otimes \cc O_{[\widetilde V/ \widetilde G]}(k ,\zeta)}\right)$ to $i_+^*\left(\fl{ \{0, \taut\} \otimes \cc O_{[\widetilde V/ \widetilde G]}(k ,\zeta)}\right)$. We show that for $0 \leq i \leq d-1$,
\[i_-^*\left(\fl{ \{0, \taut\} \otimes \cc O_{[\widetilde V/ \widetilde G]}(l+d+i ,\zeta)}\right) = i^0_*( \cc O_{BG}(l+d+i, \zeta))\]
and  \[i_+^*\left(\fl{ \{0, \taut\} \otimes \cc O_{[\widetilde V/ \widetilde G]}(l+d+i ,\zeta)}\right) = \br{\{0, \taut\} \otimes \cc O_{\cY_+}(l+d+i, \zeta) }_l.\]  The first equality is immediate from  \eqref{e:vgitex1}, the projection formula, and the fact that $(i^0)^* \circ i_-^* ( \cc O_{[\widetilde V/ \widetilde G]}(l+d+i ,\zeta)) = \cc O_{BG}(l+d+i, \zeta)$.  
Because $i \geq 0$, the summands of $\{0, \taut\} \otimes \cc O_{[\widetilde V/ \widetilde G]}(l+d+i ,\zeta)$ all have degree at least $l-d$, so the map $\{0, \taut\} \otimes \cc O_{[\widetilde V/ \widetilde G]}(l+d+i ,\zeta) \xrightarrow{rd_l} \fl{ \{0, \taut\} \otimes \cc O_{[\widetilde V/ \widetilde G]}(l+d+i ,\zeta)}$ is well-defined.
By the third equality in \eqref{e:vgitex1} and the exact triangle 
\begin{align*} \{0, \taut\} \otimes \cc O_{[\widetilde V/ \widetilde G]}(l+d+i ,\zeta) \xrightarrow{rd_l} &\fl{ \{0, \taut\} \otimes \cc O_{[\widetilde V/ \widetilde G]}(l+d+i ,\zeta)} \\ \to &\br{\{0, \taut\} \otimes \cc O_{[\widetilde V/ \widetilde G]}(l+d+i, \zeta) }_l\end{align*}
we see that \[i_+^*\left(\fl{ \{0, \taut\} \otimes \cc O_{[\widetilde V/ \widetilde G]}(l+d+i ,\zeta)}\right) = i_+^*\left( \br{\{0, \taut\} \otimes \cc O_{[\widetilde V/ \widetilde G]}(l+d+i, \zeta) }_l\right).\]  By \eqref{e:Kfrac} we deduce \eqref{e:vgitex2}.  The proof of \eqref{e:vgitex3} is exactly analogous.

\end{proof}

\begin{theorem}\label{t:Ksquare}
The map $\vgit: D(\cY_-)\xrightarrow{\sim} D(\cY_+)$ induces an equivalence 
$D(\cY_-)_{BG} \xrightarrow{\sim} D(\cY_+)_{\PP(G)}$.
Furthermore, the following diagram commutes:
\begin{equation}\label{scs}
\begin{tikzcd}
K(\cY_-)_{BG} \ar[d, "i^1_* \circ \pi_*"] \ar[r, leftrightarrow, "\vgit"] & K(\cY_+)_{\PP(G)} \ar[d, " j^* \circ \pi_*"] 
\\
K(\cXR,  w) \ar[r, leftrightarrow, "\orlov"] 
& K(\cZ) 
\end{tikzcd}
\end{equation}
where $K$ denotes the Grothendieck group of the corresponding derived category.
\end{theorem}

\begin{proof}
The first statement
can be seen by direct calculation from Propositions~\ref{p:strgenBG}, \ref{p:strgen2}, and~\ref{p:vgit}, however it follows also from general principles.
Namely, one uses the fact that $\vgit$ is the identity on the open locus where $\widetilde V$ is semi-stable with respect to both the positive and negative stability condition.  Thus if an object $\cc E_\bullet \in D(\cY_-)$ is zero on $\cY_- \setminus BG$, it will map via $\vgit$ to something which is zero on $\cY_+ \setminus \PP(G)$, and vice versa.

For the second statement consider the following diagram: 
\[
\begin{tikzcd}
K(\cY_-)_{BG} \ar[d, "i^1_* \circ \pi_*"] \ar[r, leftrightarrow, "\vgit"] & K(\cY_+)_{\PP(G)} \ar[d, " j^* \circ \pi_*"] \ar[rd, "i^1_* \circ \pi_*"]& 
\\
K(\cXR,  w) \ar[r, leftrightarrow, "\orlov"] \ar[rr,  leftrightarrow, bend right = 15, "\vgit"] & K(\cZ) \ar[r, leftrightarrow, "\tilde \phi_+"] & K(\cYR,  w).
\end{tikzcd}
\]
The bottom triangle commutes by definition and the right triangle commutes by Lemma~\ref{l:phij}.  Proving commutativity of the left square is therefore equivalent to proving commutativity of the outer square.  Finally, since $K(\cY_-)_{BG}$ and $K(\cY_+)_{\PP(G)}$ are generated by the image of $K(BG)$ and $K(\PP(G))$ respectively, it suffices to prove commutativity on the set 
\[\{i^0_*(\cc O_{BG}(l+d+i, \zeta))\}_{\zeta \in \widehat{\bar G}, 0 \leq i \leq d -1}.   \]
Observe that since $\pi_* \circ i^0_* = id$, we have $i^1_* \circ \pi_* \circ i^0_* = i^1_*$. Then by Proposition~\ref{p:vgit}
\begin{align*}
i^1_* \circ \pi_*\circ \vgit \left( i^0_*(\cc O_{BG}(l+d+i, \zeta)) \right)& = i^1_* \circ \pi_* \circ i^0_*(K(l+d+i, \zeta)) \\
& = i^1_* ( K(l+d+i, \zeta)) \\
& = \vgit \circ i^1_*(\cc O_{BG}(l+d+i, \zeta)) \\
& = \vgit \circ i^1_* \circ \pi_* \left( i^0_*(\cc O_{BG}(l+d+i, \zeta))\right).
\end{align*}

%
%
%
\end{proof}
We conclude this section with an example demonstrating that the square \eqref{scs} does not commute in the derived category.  
\begin{example}\label{countex}
Let $w$ be the Fermat quintic in five variables and let $G = \langle j \rangle$.  Then $\cY_- = [\CC^5/\mu_5]$, $\PP(G) = \PP^4$, $\cY_+ = \tot( \cc O_{\PP^4}(-5))$, and $\cZ$ is the Fermat quintic threefold.  For simplicity we take $l = 0$.  Then by Proposition~\ref{p:vgit}, $\vgitz$ maps $
i^0_* \left(\cc O_{B \mu_5}(5+i)\right)$ to $i^0_*\left( K(5+i) \right) $.  Note in this case that 
\begin{align*} K(5) &= \cc O_{\PP^4}, \\
K(6) &= \cc O_{\PP^4}^{\oplus 5} \xrightarrow{\sum x_i} \cc O_{\PP^4}(1),\end{align*}
where $K(6)$ is supported in degrees $-1$ and $0$.
Consider the morphism \[p_1: K(6) \to K(5)[1]\] given by the projection $\pi_1: \cc O_{\PP^4}^{\oplus 5} \to \cc O_{\PP^4}$ in degree $-1$.  Let 
\[q_1: \cc O_{B \mu_5}(6) = \cc O_{B \mu_5}(1) \to \cc O_{B \mu_5}(5)[1] = \cc O_{B \mu_5}[1]\]
denote the morphism $\vgitz^{-1} \circ i^0_* (p_1) \in D([\CC^5/\mu_5])$.

Observe that $j^* \circ \pi_* \circ \vgitz (q_1) = j^* \circ \pi_* \circ i^0_* (p_1) = j^*(p_1)$ is a nontrivial morphism in $D(\cZ)$.
On the other hand, 
\begin{align*}
\text{Hom}_{D(B \mu_r)} \left(\cc O_{B \mu_5}(1), \cc O_{B \mu_5}[1]\right) &= \text{Ext}^1\left(\cc O_{B \mu_5}(1), \cc O_{B \mu_5}\right)
\end{align*} is clearly zero.
So $\pi_*(q_1)$ is zero in $D(B\mu_r)$.  We conclude that $\orlovz \circ i^1_* \circ \pi_*(q_1)$ is the trivial morphism in $D(\cZ)$.   This shows that $\orlovz \circ i^1_* \circ \pi_*$ is not isomorphic to $j^* \circ \pi_* \circ \vgitz$.
\end{example}

\section{Enumerative invariants}

In this section we review the various enumerative invariants appearing in the rest of the paper.

\subsection{FJRW theory}
Let $(w,G)$ be a Landau--Ginzburg pair.  
\subsubsection{State space}

The FJRW state space is an orbifolded version of the relative cohomology with respect to $w$.  
\begin{definition}
Given $g \in G$, define
\[ \cc H_g( w, G) := H^* ( (\CC^N)^g, w^{-1}\big((m, \infty)\big); \CC)^G,\]
where $(m, \infty)$  is an interval of the real line with $m$ a sufficiently large real number.  Here we also denote by $w$ its restriction to $(\CC^N)^g$.  
The \emph{FJRW state space} is 
\[ \cc H(w, G) := \bigoplus_{g \in G} \cc H_g(w, G).\] We will sometimes denote this by $H^{(w, g)}$.
The \emph{degree} of an element $\alpha \in \cc H_g(w, G)$ is given by 
\[\deg(\alpha) := \op{rank}_{\CC}\left((\CC^N)^g\right) + 2\left(\sum_{j = 1}^N (m_j(g) - q_j)\right).\]

The \emph{narrow FJRW state space} is defined to be
\[\cc H_{\op{nar}}(w, G) := \bigoplus_{g \in G_{\op{nar}}} \cc H_g(w, G).\] We will sometimes denote this by $H^{(w, G), \op{nar}}$
\end{definition}
By Wall's isomorphism \cite[Equation~(74)]{FJR1}, the state space
$\cc H(w, G)$ is isomorphic to $HH_*(w,G) \cong \bigoplus_{g \in G} (\cc Q_{w|_{(\CC^N)^g}})^g$.  This isomorphism is canonical by \cite[Proposition~2.1]{CIR} and respects the indexing by $g\in G$ (in \cite{CIR} this is proven in the case that $G$ is cyclic, but the argument holds in our setting as well).  We will implicitly identify all these spaces in what follows.

In particular, for $g \in G_{\op{nar}}$,  $\cc H_g(w, G) = (\cc Q_{w|_{(\CC^N)^g}})^g = \CC$.  Recall that $\varphi_{g j^{-1}}$ denotes the unit generator of this summand. 
Then 
\[\cc H_{\op{nar}}(w, G) = \bigoplus_{g \in G_{\op{nar}}} \CC \varphi_{g j^{-1}}.\]

Let 
\[( - , - ): \bigoplus_{g \in G}H^* ( (\CC^N)^g, w^{-1}\big((m, \infty)\big); \CC) \times \bigoplus_{g \in G} H^* ( (\CC^N)^g, w^{-1}\big((-\infty , -m)\big); \CC) \to \CC\]
denote sum of the natural intersection pairings for each $g \in G$ (see Section~I.5.6 of \cite{ClRu} for details).  Consider the involution $I: \CC^N \to \CC^N$ given by 
\[ I(x_1, \ldots, x_N) = (\exp( \pi i c_1/d) x_1, \ldots, \exp( \pi i c_N/d) x_N).\]
Note that $w(I(x)) = -w(x)$, thus \[I_*(   H^* ( (\CC^N)^g, w^{-1}\big((m, \infty)\big); \CC) ) = H^* ( (\CC^N)^g, w^{-1}\big((-\infty , -m)\big); \CC).\]
\begin{definition}\label{FJRW pairing}
The \emph{FJRW state space pairing} is defined as follows.  Given $\alpha \in \cc H_{g_1}(w, G)$ and $\beta \in \cc H_{g_2} (w, G)$
\[
\langle \alpha , \beta\rangle^{(w, G)} = 
\tfrac{1}{|G|}(\alpha, I_*(\beta)).
\]
\end{definition}
The above definition makes sense since $\cc H_g(w, G)$ is naturally identified with $\cc H_{g^{-1}}(w, G)$.  Restricting the pairing to $\cc H_{\op{nar}}(w, G)$ yields
\[ \langle \varphi_{g_1j^{-1}}, \varphi_{g_2j^{-1}} \rangle^{(w,G)} = \frac{1}{|G|}\delta_{g_1, g_2^{-1}}.\]

\subsubsection{FJRW invariants}
Let $(\cc C, p_1, \ldots, p_n)$ denote a genus $h$, $n$-pointed stable \emph{orbi-curve} in the sense of \cite[Definition~A.1]{ChR}.  Recall (\cite[Definition~2.2.2]{FJR1}) that a (stable) \emph{W-structure} for $w$ on $(\cc C, p_1, \ldots, p_n)$ is a collection of $N$ line bundles $\cc L_1, \ldots, \cc L_N$ over $\cc C$ together with, for each monomial $m(x_1, \ldots, x_N)$ of $w$, an isomorphism 
\[
\phi_m: m(\cc L_1, \ldots, \cc L_N) \xrightarrow{\cong} \omlog,\]
where $\omlog$ is the log-canonical bundle $\omega_{\cc C}( \sum p_i)$ and $m(\cc L_1, \ldots, \cc L_N)$ is interpreted by replacing products of monomials with tensor products of line bundles.  One requires further that at each marked point $p_i$, the induced representation 
\[ G_{\rho_i} \to G_{max}(w) \subset GL_N(\CC)\]
is faithful.  

Note that the group $G$ does not enter into the definition of a W-structure for $w$.  To incorporate the group,
choose a Laurent polynomial $\tilde  w$ such that $G_{max}(\tilde w) = G$.  Then a W-structure for the pair $(w, G)$ is simply a W structure for $\tilde w$.

\begin{definition}\label{d:wstr}  Given $h$ and $n$ such that $2h-2 + n \geq 0$.  The moduli of W-structures of genus $h$ with $n$ marked points for $(w, G)$  is denoted 
$\sWbar_{h, n}(G)$.  \end{definition}
The moduli space $\sWbar_{h, n}(G)$ is a smooth Deligne--Mumford stack by \cite[Theorem~2.2.6]{FJR1}.

The moduli of W-structures for $(w,G)$ is a union of open and closed substacks indexed by the action of the isotropy group at the marked points $p_1, \ldots, p_n$.  Given a W structure on the curve $(\cc C, p_1, \ldots, p_n)$, the isotropy group $G_{\rho_i}$ at $p_i$ is equal to $\langle \omega_{r_i} \rangle$ where $\omega_{r_i} = e^{2 \pi i/r_i}$ acts by rotation by $2 \pi/r_i$ on the tangent space $T_{p_i} \cc C$.  Then the image of $\omega_{r_i}$ under the inclusion $G_{\rho_i} \to G$ is an element $g_i \in G$.  Given an $n$-tuple of group elements $\bg = (g_1, \ldots, g_n)$, let 
\[\sWbar_{h, \bg}(G) \subset \sWbar_{h, n}(G)\]
denote the open and closed subset where, for $1 \leq i \leq n$, the generator of the isotropy group at $p_i$ maps to $g_i$ under the inclusion $G_{\rho_i} \to G$.

Given a W-structure on a smooth orbi-curve $\cc C$,  let $|\cc L_j|$ denote the rigidified line bundle on the coarse curve $|\cc C|$.  Assume that the generator of the isotropy group at $p_i$ is indexed by $g_i$.
We record the following useful fact (Equation (18) of \cite{FJR1}):
\begin{equation}\label{rig deg}
\deg(|\cc L_j|) = c_j/d(2h - 2 + n) - \sum_{i = 1}^n m_j(g_i).
\end{equation}
\begin{remark}
The fact that \eqref{rig deg} is an integer places conditions on when the moduli space $\sWbar_{h, \bg}(G)$ will be non-empty.
\end{remark}

In \cite[Section~4.1.1]{FJR1}, a virtual cycle $[\sWbar_{h, \bg}(G)]^{vir}$ is constructed, which defines a map
\[
[\sWbar_{h, \bg}(G)]^{vir} \cap - : \bigotimes_{i = 1}^n \cc H_{g_i}(w, G) \to H_*( \sWbar_{h, \bg}(G)).\]
Pushing this class forward to $\sMbar_{h,n}$ and taking the Poincare dual defines a cohomological field theory \cite[Theorem~4.2.2]{FJR1}  with state space $\cc H(w, G)$.  Psi-classes on $\sWbar_{h, \bg}(G)$ can be defined by pulling back from $\sMbar_{h,n}$.  

\begin{definition}
Given elements $\alpha_i \in \cc H_{g_i}(w, G)$, integers $b_i \geq 0$ for $1 \leq i \leq n$, and an integer $h \geq 0$ with $2h-2 + n > 0$, define the genus $h$ FJRW invariant $\langle \alpha_1 \psi^{b_1}, \ldots, \alpha_n \psi^{b_n} \rangle^{(w, G)}_{h,n}$ as
\[
(-1)^{\chi(\oplus \cc L_i)}  
\int_{[\sWbar_{h, \bg}(G)]} [\sWbar_{h, \bg}(G)]^{vir} \cap \left(\alpha_1 \otimes \cdots \otimes \alpha_n\right) \cap \left( \prod_{i=1}^n \psi_i^{b_i} \right).\]
Invariants can be defined for general classes $\alpha_i \in \cc H(w, G)$ by extending linearly.
\end{definition}

\begin{remark}\label{factor mod}
The FJRW invariants defined above differ from those originally defined in \cite{FJR1} by a factor of 
$(-1)^{\chi(\oplus \cc L_i)} 
|G|^h / \deg ( \sWbar_{h, \bg}(G) \to \sMbar_{h, n})$.  This modification factor 
 was also removed in \cite{CIR}.
 In \cite[Appendix~B]{CIR}, they show that the modification by $|G|^h / \deg ( \sWbar_{h, \bg}(G) \to \sMbar_{h, n})$ still defines a cohomological field theory with respect to the pairing given in Definition~\ref{FJRW pairing}.  
The further modification by $(-1)^{\chi(\oplus \cc L_i)}$ defines a cohomological field theory if we adjust our pairing by
\[\langle \alpha, \beta \rangle^{(w, G)} \mapsto (-1)^{\dim (\CC^N)^g} \langle \alpha, \beta \rangle^{(w, G)}\]
for $\alpha \in \cc H_g(w, G)$.  Note that this does not affect the narrow FJRW state space.
\end{remark}

\subsection{Gromov--Witten theory}
\begin{definition}
Given $\cc X$  a smooth Deligne--Mumford stack.  Let $\sMbar_{h,n}(\cc X, d)$ denote the moduli space of representable degree $d$ stable maps from orbi-curves of genus $h$ with $n$ marked points.  Here $d$ is an element of 
 the cone $\text{Eff} = \text{Eff}(\cX) \subset H_2(\cX, \QQ)$ of effective curve classes.   Let $[\sMbar_{h,n}(\cc X, d)]^{vir}$ denote the virtual fundamental class of \cite[Definition~5.2]{BF} and \cite[Section~4.5]{AGV}. 
\end{definition}

Assume $\cX$ is proper.  Let $I\cX$ denote the inertia stack of $\cX$, defined as the fiber product of the diagonal morphism with itself (see \cite{ChenR1} for details).  The inertia stack is a union of connected components called \emph{twisted sectors}:
\[I\cX = \coprod_{b \in B} \cX_b.\]  There is a distinguished component isomorphic to $\cX$ itself, called the untwisted sector.   Let 
\[( - , - ): H^*( I\cX; \CC) \times H^*( I\cX; \CC) \to \CC\]
denote the pairing defined by integration.  The universal property of the fiber product induces a natural involution which we (again) denote by \[I: I\cX \to I\cX.\]

If $\cX$ is a quotient of the form $[V / \Gamma]$ where $V$ is a smooth variety and $\Gamma$ is an abelian group, then 
\[ I \cX = \coprod_{\gamma \in \Gamma} [ V^\gamma / \Gamma].\]
Denoting the $\gamma$-twisted sector $[V^\gamma /\Gamma]$ by ${\cX}_\gamma$,
in this case the involution $I$ maps ${\cX}_\gamma$ to ${\cX}_{\gamma^{-1}}$ via the natural isomorphism.  

Given a twisted sector $\cX_b$, recall as in Section~\ref{s:chch} that the restriction 
 $T_b\cX$ of the tangent bundle $T\cX$ to 
the twisted sector splits into a sum of eigenbundles $T_{b,f}\cX$ for $0 \leq f <1$ based on the weight of the generator of the isotropy at a geometric point $(x, g_b)$.
\begin{definition}  If $B$ indexes the twisted sectors $\cX_b$ of $I\cX$,
define the \emph{age} of the $b$th sector to be
\[ \age_b := \sum_{0 \leq f < 1} f \dim_\CC(T_{b,f}\cX).\]
\end{definition}
%
%

\begin{definition}\cite[Definition~3.2.3]{ChenR1}
The \emph{Chen--Ruan} cohomology of $\cX$ is the $\QQ$-graded ring \[H^*_{\op{CR}}(\cX) := \bigoplus_{n \in \QQ} H^n_{CR}(\cX)\] where
\[ H^n_{CR}(\cX) := \bigoplus_{ b \in B} H^{n - 2 \age_b}(\cX_b, \CC).\]
We will also refer to this as the  \emph{GW state space} of $\cX$ and sometimes denote it by $H^\cX$.  Define the \emph{GW state space pairing} to be
\[ \langle \alpha, \beta \rangle^\cX := (\alpha, I_*(\beta) ).\]
\end{definition}

Let $\bar I \cX$ denote the \emph{rigidified inertia stack} as in \cite[Section~3.4]{AGV}.  Recall that there exist evaluation maps $ev_i: \sMbar_{h, n}(\cX, d)  \to \bar I \cX$ for each marked point $p_i$.  By the discussion in section 6.1.3 of \cite{AGV}, it is convenient to work as if the map $ev_i$ factors through $I \cX$.  While this is not in fact true, due to the isomorphism $H^*(\bar I \cX, \CC) \cong H^*(I\cX, \CC)$, it makes no difference in terms of defining Gromov--Witten invariants.  See Section 6.1.3 of \cite{AGV} for details.

As in the case of FJRW theory, the moduli space $\sMbar_{h, n}(\cX, d)$ splits into open and closed components based on the action of the isotropy group at each marked point.  Given an $n$-tuple $\bg = (g_1, \ldots, g_n)$ indexing $n$ components of $I\cX$, let 
$\sMbar_{h, \bg}(\cX, d)$ denote the open and closed subset of $\sMbar_{h, n}(\cX, d)$ such that $ev_i$ maps to ${\cX}_{g_i}$ for $1 \leq i \leq n$.

If $\cX$ is proper, then so is $\sMbar_{h,n}(\cc X, d)$ for any genus $h$ and degree $d$, by Theorem~1.4.1 of \cite{AV}.  For $1 \leq i \leq n$, let $\psi_i \in c_1(\mathbb L_i)$, where $\mathbb L_i$ is the line bundle on $\sMbar_{h,n}(\cc X, d)$ whose fiber over a stable map $f: \cc C \to \cX$ is $T_{p_i}|\cc C|$, the cotangent space at the $i$th marked point of the coarse underlying space of the source curve, $| \cc C|$.
\begin{definition}
  Given $\alpha_1, \ldots, \alpha_n \in H^*_{\op{CR}}(\cX)$ and integers $b_1, \ldots, b_n \geq 0$ define the Gromov--Witten invariant
\[ \langle \alpha_1 \psi^{b_1}, \ldots, \alpha_n \psi^{b_n} \rangle^{\cX}_{h, n, d} := 
\int_{[\sMbar_{h,n}(\cc X, d)]^{vir}} \prod_{i=1}^n ev_i^*( \alpha_i) \cup \psi_i^{b_i}.\]
\end{definition}

Next we consider the non-proper case.  
\begin{definition}\cite[Definition~2.13]{Sh1} \label{d:narcoh}
Let $\cY$ be a non-proper space.  Let $H^*_{\op{CR}, \op{c}}(\cY)$ denote cohomology with compact support and consider the natural map $\iota: H^*_{\op{CR}, \op{c}}(\cY) \to H^*_{\op{CR}}(\cY)$.  We define the \emph{narrow} part of $H^*_{\op{CR}}(\cY)$ to be the image of this map:
\[ H^*_{\op{CR}, \op{nar}}(\cY) := \text(im)(\iota).\] This is referred to as the \emph{narrow GW state space} and will sometimes be denoted by $H^{\cY, \op{c}}$.
We define a pairing on $H^*_{\op{CR}, \op{nar}}(\cY)$ as follows.  Given $\alpha, \beta \in H^*_{\op{CR}, \op{nar}}(\cY)$, choose $\tilde \alpha$ to be some class in $H^*_{\op{CR}, \op{c}}(\cY)$ such that $\iota(\tilde \alpha) = \alpha$.  Then, if $( - , - ): H^*_{\op{CR}, \op{c}}(\cY) \times H^*_{\op{CR}}(\cY) \to \CC$ is the usual pairing between cohomology and cohomology with compact support and $I: I\cY \to I\cY$ is the involution map as above, define
\[ \langle \alpha, \beta \rangle^{\cY, \op{nar}} := (\tilde \alpha, I_*(\beta)).\]
\end{definition}

\begin{lemma}\cite[Proposition~2.11 and Section~2.2]{Sh1}
The pairing $\langle - , - \rangle^{\cY, \op{nar}}: H^*_{\op{CR}, \op{nar}}(\cY) \times H^*_{\op{CR}, \op{nar}}(\cY) \to \CC$ is well-defined and nondegenerate.
\end{lemma}
By the same reasoning used to define the pairing, one can define a \emph{compactly supported} cup product between elements in $H^*_{\op{CR}, \op{nar}}(\cY)$.  Given $\alpha, \beta \in H^*_{\op{CR}, \op{nar}}(\cY)$, define
\begin{align}\label{e:cprod}
\alpha \cup_c \beta := \tilde \alpha \cup \beta \in H^*_{\op{CR}, \op{c}}(\cY).
\end{align}
where $\tilde \alpha \in H^*_{\op{CR}, \op{c}}(\cY)$ is a lift of $\alpha$.

If $\cY$ is not proper, there is still a well-defined virtual class, and so Gromov--Witten invariants may still be defined by capping with classes of compact support.  In particular assume that the evaluation maps $ev_i$  are proper.  Then $ev_i^*(\alpha_i)$ lies in cohomology with compact support whenever $\alpha_i  \in H^*_{\op{CR}, \op{c}}(\cY)$ for some $1\leq i \leq n$.  Alternatively the cup product $ev_i^*(\alpha_i)\cup_c ev_j^*(\alpha_j)$ (see \eqref{e:cprod}) lies in cohomology with compact support whenever $\alpha_i, \alpha_j  \in H^*_{\op{CR}, \op{c}}(\cY)$ for some $1\leq i < j \leq n$.  In either case the Gromov--Witten invariant 
$ \langle \alpha_1 \psi^{b_1}, \ldots, \alpha_n \psi^{b_n} \rangle^{\cY}_{h, n}$
will be well-defined.

\begin{lemma}[Lemma 3.8 of \cite{IMM} and  Lemma 4.2 of \cite{Sh1}] \label{l:evprop} Let $\cX$ be a smooth and proper Deligne--Mumford stack and let $\cc E \to \cX$ be a vector bundle on $\cX$.  Let $\cY$ denote the total space of $\cc E^\vee$.  For $1 \leq i \leq n$, the evaluation maps $ev_i: \sMbar_{0, n}(\cY, d)  \to \bar I \cY$ are proper if either 
\begin{enumerate}
\item The degree $d = 0$; 
\item The vector bundle $\cc E$ is pulled back from a vector bundle $E \to X$ on the coarse space and $E$ is convex. 
\end{enumerate}
\end{lemma}

%
%
%
%
%
%
%

\section{Structures}\label{s:structures}


Let $(w, G)$ be a Landau--Ginzburg pair and let
$\{T_i\}_{i \in I}$ be a basis for the state space $\cc H(w, G)$ then we may denote by $\bt := \sum t^i T_i$ a point in $\cc H(w, G)$.
Given $\alpha_1, \ldots, \alpha_n \in \cc H(w, G)$ and integers $b_1, \ldots, b_n \geq 0$,
define
\begin{equation}\label{e:fps1}\langle \langle \alpha_1 \psi^{b_1}, \ldots, \alpha_n \psi^{b_n} \rangle \rangle^{(w, G)}(\bt) := \sum_{k\geq 0}\frac{1}{k!} \langle \alpha_1 \psi^{b_1}, \ldots, \alpha_n \psi^{b_n}, \bt, \ldots, \bt \rangle^{(w, g)}_{0,n+k}\end{equation}
where a summand is implicitly assumed to be zero if $n+k < 3$.  
Equation
 \eqref{e:fps1} is a 
formal power series is lying in $\CC[[\bt]]:= \CC[[t^i]]_{i \in I}$.

Similarly, given a smooth Deligne--Mumford stack $\cY$, 
choose a basis $\{T_i\}_{i \in I}$ for $H^*_{\op{CR}}(\cY)$ such that $I = I' \coprod I''$ where $I''$ indexes a basis for the subspace $H^2(\cY)$ in the state space and $I'$ indexes a basis for 
\[H^{\neq 2}(\cY) \oplus \bigoplus_{b \in B| b \neq \op{id}} H^*(\cY_b)\]
where the latter direct sum is over all twisted sectors other than the untwisted sector.
Let $\bt ' = \sum_{i \in I '} t^i T_i$ and let $\bt = \sum_{i \in I' \cup I''} t^i T_i$.  Let $q^i = e^{t_i}$ for $i \in I''$.
 Define
\begin{equation}\label{e:fps2}\langle \langle \alpha_1 \psi^{b_1}, \ldots, \alpha_n \psi^{b_n} \rangle \rangle^\cY(\bt) := \sum_{d \in \text{Eff}}\sum_{k\geq 0}\frac{1}{k!} \langle \alpha_1 \psi^{b_1}, \ldots, \alpha_n \psi^{b_n}, \bt, \ldots, \bt \rangle^\cY_{0,n+k,d}\end{equation}
where a summand is implicitly assumed to be zero if $d = 0$ and $n+k < 3$.    
Let $\CC[[\bt ']] := \CC[[ t^i]]_{i \in I ' }$ and $\CC[[\bq]] := \CC[[q^i]]_{i \in I ''}$.
Via the divisor equation (see \cite{AGV}), when $b_1 = \cdots = b_n = 0$ expression~\eqref{e:fps2} can be viewed as a formal power series in $\CC[[\bt ', \bq]]$.

If $\cY$ is not proper, the expression is still well defined assuming either one of the $\{\alpha_i\}$ lies in $H^*_{\op{CR}, \op{c}}(\cY)$ or two of the $\{\alpha_i\}$ lie in $H^*_{\op{CR}, \op{nar}}(\cY)$ and that the evaluation maps are proper as is the case in Lemma~\ref{l:evprop}.
 
 \begin{notation}\label{notb}
 To deal simultaneously with the case of FJRW or GW theory, we will let $P^\square$ denote $\CC[[\bt]]$ in the case of FJRW theory $\square = (w, G)$ and $\CC[[\bt ', \bq]][\ln(\bq)]$ in the case of GW theory $\square = \cY$.  In either case we will denote the state space by $H^\square$.
 \end{notation}

\subsection{Quantum $D$-modules}

Let $\square = \cX$ or $(w, G)$ denote either the Gromov--Witten theory of a smooth Deligne--Mumford stack $\cX$ with projective coarse moduli space or the FJRW theory of the Landau--Ginzburg model $(w, G)$.  Choose a basis $\{T_i\}_{i \in I}$  for the state space  $H^\square$ and let $\bt = \sum_{i \in I} t^i T_i$.  Define the quantum product as follows.  For elements $\alpha, \beta, \gamma$ in the state space $H^\square$, define $\alpha \bullet_{\bt}^\square \beta$ by
\[ \langle \alpha \bullet_{\bt}^\square \beta, \gamma \rangle^\square = \langle \langle \alpha, \beta, \gamma\rangle \rangle^\square(\bt).\]  The coefficients lie in $P^\square$.

\begin{definition}
Define the \emph{central charge} to be
\[\hat c^{\cX} := \dim(\cX)\]
in the case of Gromov--Witten theory and
\[\hat c^{(w, G)} := N - 2 \left(\sum_{j=1}^N q_j \right)\]
in the case of FJRW theory.
\end{definition}

The pairing $\langle - , - \rangle^\square$ on $H^\square$ can be extended to a $z$-sesquilinear pairing $S^\square$ on $H^\square \otimes P^\square[z, 1/z]$ by defining
\[ S^\square( u(z), v(z)) := (2 \pi i z)^{\hat c^\square} \langle u(e^{\pi i} z), v(z) \rangle^\square.\] 
The quantum connection is defined by the formula 
\[ \nabla_i^\square  = \partial_i + \frac{1}{z} T_i \bullet_{\bt}^\square.\]
We can define the quantum connection in the $z$-direction as well.  Define the Euler vector field
\[ \cc E := \partial_\rho+ \sum_{i \in I} \left(1 - \frac{1}{2} \deg T_i\right) t^i \partial_i\]
where $\rho:= c_1(T\cX) \subset H^2(\cX)$ in the case of Gromov--Witten theory and is defined to be zero in FJRW theory.
Define the grading operator $\op{Gr}$ by 
\[\op{Gr}(T_i) := \frac{\deg T_i}{2} T_i.\]
The quantum connection can be extended to the $z$-direction as
\[ \nabla_{z}^\square := \partial_z - \frac{1}{z^2} \cc E \bullet^\square_{\bt} + \frac{1}{z} \op{Gr}.\]

\begin{definition}
The \emph{quantum $D$-module} QDM($\square$) is defined to be the triple:
\[QDM(\square) := \left( H^\square \otimes P^\square[z, 1/z], \nabla^\square, S^\square \right).\]
\end{definition}

We define $L^\square \in \op{End}(H^\square) \otimes P^\square[z^{-1}]$ by 
\[L^\square(\bt, z) \alpha := \alpha + \sum_{i \in I} \langle \langle \frac{\alpha}{z - \psi}, T_i \rangle \rangle^\square(\bt) T^i.\]

The following is well-known.  See, for instance, \cite[Proposition~2.4]{Iri} and \cite[Proposition~2.7]{CIR}.
\begin{theorem}
The quantum connection $\nabla^\square$ is flat, with fundamental solution $L^\square(\bt, z) z^{-{\op{Gr}}}z^\rho$:
\begin{equation} \nabla_i \left(L^\square(\bt, z) z^{- \op{Gr}}z^\rho \alpha\right) = \nabla_z \left(L^\square(\bt, z) z^{- \op{Gr}} z^\rho\alpha\right) = 0 \end{equation}  for  $i \in I$ and $\alpha \in H^\square$.  Furthermore the pairing $S^\square$ is $\nabla^\square$-flat and for $\alpha, \beta \in H^\square$,
\begin{equation}\label{e:SL} S(L^\square(\bt, z) \alpha, L^\square(\bt, z) \beta) = \langle \alpha, \beta\rangle^\square.\end{equation}
\end{theorem}

\subsubsection{Special cases}\label{ss:spec}

We detail below three important quantum $D$-modules arising by restricting the state space in particular ways.

(1) Ambient GW theory:

Let $\cX$ be a smooth and proper DM stack.  Let $\cE$ be a convex vector bundle on $\cX$, and let $\cZ$ be the zero locus of a transverse section $s \in \Gamma( \cX, \cE)$.  Let $j: \cZ \to \cX$ denote the inclusion map and define $H^*_{\op{CR}, \op{amb}}(\cZ) := \op{im} (j^*)$.
\begin{assumption}\label{a:ass2}
We  assume that the Poincar\'e pairing on $H^*_{\op{CR}, \op{amb}}(\cZ)$ is non-degenerate.  
\end{assumption}
This holds for instance if $\cE$ is the pullback of an ample line bundle on the coarse space 
$X$ by Lemma~3.14 and Remark~3.15 of \cite{IMM}.
Although the argument there is for manifolds, the orbifold statement follows almost immediately after noting that in this case the restriction of $\cE$ to each twisted sector will still be the pullback of an ample line bundle on the coarse space of the twisted sector.
\begin{proposition}[Corollary~2.5 of \cite{Iri3} and Remark~3.13 of \cite{Sh1}]
For $\bar \bt \in H^*_{\op{CR}, \op{amb}}(\cZ)$, the quantum product $\bullet^\cZ_{\bar \bt}$ is closed on $H^*_{\op{CR}, \op{amb}}(\cZ)$.  The quantum connection and solution $L^\cZ(\bar \bt, z)$ preserve $H^*_{\op{CR}, \op{amb}}(\cZ)$ for $\bar \bt \in H^*_{\op{CR}, \op{amb}}(\cZ)$.
\end{proposition}
%
\begin{definition}
The \emph{ambient} quantum $D$-module is defined to be 
\[ QDM_{\op{amb}}(\cZ) := (H^*_{\op{CR}, \op{amb}}(\cZ) \otimes \CC[[\bar \bt ', \bar \bq]][z, z^{-1}], \nabla^\cZ, S^\cZ)\]
where $\bar t^i$ denotes dual coordinates to a basis $\{T_i\}$ of  $H^*_{\op{CR}, \op{amb}}(\cZ)$.
\end{definition}

(2) Narrow FJRW theory:

An analogous statement holds for the narrow part of FJRW theory, at least when $w$ is a Fermat polynomial.
\begin{proposition}  Assume $w$ is Fermat.  
For $\bar \bt \in \cc H_{\op{nar}}(w, G)$, the quantum product $\bullet^{(w,G)}_{\bar \bt}$ is closed on $\cc H_{\op{nar}}(w, G)$.  The quantum connection and solution $L(\bar \bt, z)$ preserve $\cc H_{\op{nar}}(w, G)$ for $\bar \bt \in \cc H_{\op{nar}}(w, G)$.
\end{proposition}

\begin{proof}
Choose a basis $\{T_i\}_i \in I$ of $\cc H(w, G)$ such that each element $T_i$ lies in a particular summand $\cc H_{g_i}(w, G)$ for some $g_i \in G$.  Then $T_i \in \cc H_{\op{nar}}(w, G)$ if and only if the dual element $T^i \in \cc H_{\op{nar}}(w, G)$.  To prove the proposition, it then suffices to show that genus zero FJRW invariants with exactly one broad insertion must always vanish.
Namely,
if  $g_1, \ldots, g_{n-1}$ lie in $G_{\op{nar}}$, and $\theta_{g_n}$ lies in $\cc H_{g_n}(W,G)$ where $g_n$ is broad, then \[\langle \varphi_{g_1j^{-1}} \psi^{b_1}, \ldots, \varphi_{g_nj^{-1}} \psi^{b_n} \rangle^{(w, G)}_{0,n} = 0.\]
This fact is Proposition 2.8 of \cite{CIR}.  Although the proposition cited is stated only for the case where $G = \langle j \rangle$, the proof of that proposition is valid whenever $w$ is Fermat, and does not depend on the chosen group $G$ of symmetries of $w$.
%
\end{proof}

\begin{definition}
The \emph{narrow} quantum $D$-module for $(w, G)$ is defined to be 
\[ QDM_{\op{nar}}(w, G) := (\cc H_{\op{nar}}(w, G) \otimes \CC[[\bar \bt ]][z, z^{-1}], \nabla^{(w,G)}, S^{(w,G)})\]
where $\bar t^i$ denotes dual coordinates to a basis $\{T_i\}$ of  $\cc H_{\op{nar}}(w, G)$.
\end{definition}

(3) Narrow Gromov--Witten theory:
Let $\cY$ be a non-proper space, but assume that all evaluation maps are proper.  Then there is a  quantum product, defined similarly as in the proper case.  For $\alpha, \beta \in H^*_{\op{CR}}(\cY)$ and $\gamma \in H^*_{\op{CR}, \op{c}}(\cY)$,
\[\langle \alpha \bullet_{\bt} \beta, \gamma\rangle = \langle \langle \alpha, \beta, \gamma\rangle \rangle^{\cY}(\bt).\]
Note that $\alpha \bullet_{\bt} \beta$ still lies in $H^*_{\op{CR}}(\cY)$.  One can show (see Section~4.1 of \cite{Sh1}) that the quantum connection $\nabla^\cY$ and solution $L^\cY$ are also well defined.

\begin{proposition}\cite[Proposition~4.6]{Sh1} \label{p:narpres}
The narrow state space is preserved by the quantum product. 
The quantum connection $\nabla^\cY$ and solution $L^\cY( \bt, z)$ preserve $ H^*_{\op{CR}, \op{nar}}(\cY)$ for all $\bt \in H^*_{\op{CR}}(\cY)$. 
\end{proposition}

\begin{notation}
When we are explicitly restricting to the narrow state space of Definition~\ref{d:narcoh} we will  denote the quantum connection and solution by $\nabla^{\cY, \op{nar}}$ and solution $L^{\cY, \op{nar}}( \bt, z)$.  Note however that $\bt$ can still range over all of $H^*_{\op{CR}}(\cY)$.
\end{notation}

\begin{definition}\label{d:nard}
The \emph{narrow} quantum $D$-module for $\cY$ is defined to be 
\[ QDM_{\op{nar}}(\cY) := (H^*_{\op{CR}, \op{nar}}(\cY) \otimes \CC[[\bt, \bq ]][z, z^{-1}], \nabla^{\cY}, S^{\cY})\]
where $ t^i$ denotes dual coordinates to a basis $\{T_i\}$ of  $H^*_{\op{CR}}(\cY)$.
\end{definition}

%
%
%

\subsection{Integral structure}\label{ss:is}

In \cite{Iri}, Iritani defines an integral structure for GW theory.  This is extended in \cite{CIR} to FJRW theory as well.  We recall the ingredients here.  

\subsubsection{Gamma class}

\begin{definition}
For $\cE$ a vector bundle on $\cY$, define the \emph{Gamma class} $\hat \Gamma (\cE)$ to be the class in $H^*_{\op{CR}}(\cY)$:
\[ \hat \Gamma (\cE)  := \bigoplus_{b \in B} \prod_{0 \leq f < 1} \prod_{i = 1}^{\op{rk}(\cE_{b,f})} \Gamma(1 - f + \delta_{b,f,i})\]
where $\Gamma(1 + x)$ should be understood in terms of its Taylor expansion at $x = 0$ and $\{\delta_{b,f,i}\}$ are the Chern roots of $\cE_{b,f}$.  Define $\hat \Gamma_{\cY}$ to be the class $\hat \Gamma(T\cY)$.
\end{definition}


\begin{definition}
The Gamma class in FJRW theory is defined to be the operator
\[\hat \Gamma_{(w, G)} := \bigoplus_{g \in G} \prod_{j=1}^N \Gamma(1 - m_j(g))\]
where the $g$th summand acts on the $g$th sector of $\cc H(w,G)$ by scaling by the given factor.
\end{definition}

\subsubsection{Flat sections}
\begin{definition}
Define the operator $\deg_0$ to be the degree operator which multiplies a homogeneous class by its unshifted degree.  In FJRW theory on the narrow state space this is multiplication by $-2 \left(\sum q_j \right)$.  In Gromov--Witten theory it multiplies a class in $H^n(I\cY)$ (with the standard grading) by $n$.
\end{definition}

Let $D^\square$ denote $D(\cY)$ in the case of GW theory or $D(\cXR, w)$ in the case of FJRW theory.  Given an object $E \in D^\square$ define
\[\bs^\square(E)(\bt, z) := \frac{1}{(2 \pi i)^{\hat c^\square}} L^\square(\bt, z) z^{-Gr} z^{\rho} \hat \Gamma_{\square} \left( (2 \pi i)^{\deg_0/2} I^*(\ch(E))\right),\]
lying in $H^\square \otimes P^\square[z, z^{-1}]$.  
Note that $\bs^\cY$ is defined even when $\cY$ is not proper. 

\begin{proposition}[Proposition~2.2.1 of \cite{CIR} and Section~3.2 of \cite{Iri3}]\label{p:pair}
The map $\bs^\square$ identifies the pairing in the derived category with $S^\square$ up to a sign.  

Let $(w, G)$ be a Landau--Ginzburg pair, where $w$ is a polynomial in $N$ variables.  Then
\[ S^{(w, G)} (\bs(E)(\bt, z), \bs(F)(\bt, z)) = e^{\pi i (N + \sum q_j)}\chi(F, E)\] 
for all $E, F \in D(\cY_-, w)$.

Let $ \cX$ be a smooth and proper Deligne--Mumford stack of dimension $N$. Then
\[ S^\cX (\bs(E)(\bt, z), \bs(F)(\bt, z)) = e^{\pi i N} \chi(F, E)\] 
for all $E, F \in D(\cX)$.
\end{proposition}
\begin{proof}
The first statement in the Calabi--Yau case (where $\sum_j q_j = 1$) assuming $G = \langle J \rangle$ is Proposition 2.21 of \cite{CIR}.  The proof is the same in the general case.  We note a minor typo in the proof; the final sentence should say this equals  $e^{\pi i (N + \sum q_j)}\chi(F, E)$, not $e^{\pi i (N + \sum q_j)}\chi(E, F)$.
The second statement is Property (iii) after Definition 3.6 in \cite{Iri3}.
\end{proof}

\begin{definition}
Assuming that $H^\square$ is spanned by the image of the Chern map $\ch: D^\square \to H^\square$,
the set 
\[ \{ \bs^\square(E)(\bt, z) | E \in D^\square\}\] forms a lattice in $H^\square \otimes P^\square[z, z^{-1}]$.   This is the \emph{integral structure} for $QDM(\square)$.  
\end{definition}
\begin{remark}
Although the assumption that $H^\square$ is spanned by the image of the Chern map will rarely hold for an arbitrary orbifold or LG space, it will turn out to hold in the restricted \emph{ambient} and \emph{narrow} setting.
\end{remark}

In the case where $\square = (\cZ, \op{amb})$, $((w, G), \op{nar})$, or $(\cY, \op{nar})$ is one of the quantum $D$-modules obtained by restricting the state space, we may restrict the integral structure of the larger theory:
\begin{align*}
\{ \bs^{\cZ, \op{amb}}(E)(\bt, z) &| E \in j^*(D(\cX)), j: \cZ \to \cX \} \\
\{ \bs^{(w, G), \op{nar}}(E)(\bt, z) &| E \in i^1_*(D(BG)) \} \\
\{ \bs^{\cY, \op{nar}}(E)(\bt, z) &| E \in D_{\op{c}}(\cY)  \}.
\end{align*}
We note in the latter example that if $E \in D_{\op{c}}(\cY)$ then $\ch(E) \in H^*_{\op{CR}, \op{nar}}(\cY)$ by \cite[Proposition~8.5]{Sh1}.  In each of these restricted theories the analogue of Proposition~\ref{p:pair} holds.

\subsection{The Integral structure for $\cY_{+/-}$}  
Recall the definition of $\cY_-$ and $\cY_+$ \eqref{e:GIT1} from our initial setup.  In both of these non-compact examples, one may define the integral structure a different way; by restricting the derived category to an appropriate subcategory of objects with proper support.  In particular, $\cY_-$ and $\cY_+$ may be viewed as the total space of a vector bundle over a proper base.  Therefore we may define the integral structure for $QDM_{\op{nar}}(\cY_-)$ to be 
\[\{ \bs^{\cY_-}(E)(\bt, z) | E \in D(\cY_-)_{BG}\},\]
and define the integral structure for $QDM_{\op{nar}}(\cY_+)$ to be 
\[\{ \bs^{\cY_+}(E)(\bt, z) | E \in D(\cY_+)_{\PP(G)}\}.\]
In both cases the quantum $D$-module is spanned by these lattices, as the following lemmas show.

\begin{lemma}\label{l:chgen}
The image of the map $\ch \circ i^0_*: D(BG) \to H^*_{\op{CR}}(\cY_-)$ generates a lattice spanning $H^*_{\op{CR}, \op{nar}}(\cY_-)$, consequently the narrow quantum $D$-module of $\cY_-$ is spanned by $ \{ \bs^{\cY_-}(E)(\bt, z) | E \in D(\cY_-)_{BG}\}$.
\end{lemma}

\begin{proof}
Because $L^\square$ is invertible, the second statement follows from the first.  Note that the map 
\[ \sum_{g \in G} c_g \ii_g \mapsto \sum_{g \in G_{\op{nar}}}\bigoplus_{g \in G} c_g \cdot \prod_{j=1}^N  \left( 1 - e^{2 \pi i (-m_j(g))} \right) \ii_g\]
surjects onto $\bigoplus_{g \in G_{\op{nar}}} \CC \cdot \ii_g$.  
Let $\{ \rho_i\}_{i \in I}$ denote the set of characters of $G$, each corresponding to a line bundle $\CC_{\rho_i}$ over both $BG$ and $\cY_-$.  Since $G$ is abelian, there are $|G|$ distinct characters.  The set $\{\ch(\CC_{\rho_i})\}_{i \in I}$ is linearly independent and therefore generates $H^*_{\op{CR}}(BG)$.
By \eqref{e:chzerobeta} and the projection formula the first statement follows.
\end{proof}

\begin{lemma}
The image of the map $\ch \circ i^0_*: D(\PP(G)) \to H^*_{\op{CR}}(\cY_+)$ generates a lattice spanning $H^*_{\op{CR}, \op{nar}}(\cY_+)$, consequently the narrow quantum $D$-module of $\cY_+$ is spanned by $ \{ \bs^{\cY_+}(E)(\bt, z) | E \in D(\cY_+)_{\PP(G)}\}$.
\end{lemma}

\begin{proof}
The argument is analogous to that of the previous lemma.  By Proposition 2.14 in \cite{Sh1}, $H^*_{\op{CR}, \op{nar}}(\cY_+) = 
i^0_* (H^*_{\op{CR}}(\PP(G)))$. 
It is left to check that $\ch (D(\PP(G)))$ generates $H^*_{\op{CR}}(\PP(G))$.  By \cite[Theorem~5.3]{BH}, there is a combinatorial Chern character map which yields an isomorphism
\[ \op{c-ch}: K_0( \PP(G)) \xrightarrow{\sim} H^*_{\op{CR}}(\PP(G)).\]
We must check that this map is equal to the orbifold Chern character of Definition~\ref{d:orbch}.

For $\Sigma$ a fan corresponding to a projective toric DM stack, the main result of \cite{BH} implies that $K_0( \cX(\Sigma))$ is generated by toric boundary divisors.  
Let $\rho$ be a ray of $\Sigma$, corresponding to the line bundle $\cc O(D_{\rho})$ on $\cX(\Sigma)$.  Let $v$ be an element of $\op{Box}(\Sigma)$ corresponding to a component ${\cX}_v$ of the inertia stack $I\cX$.  
The isomorphism $\op{c-ch}$, when composed with the projection $H^*_{\op{CR}}(\cX) \to H^*({\cX}_v)$, maps $\log(e^{-2 \pi i f} \cc O( D_\rho))$ to $c_1(\cc O(D_{\rho}))$, where $e^{2 \pi i f}$ is the action of the isotropy group on $\cc O(D_{\rho})$.  Therefore $\cc O(D_{\rho})$ is mapped to 
\[\op{c-ch}(\cc O(D_{\rho}))|_v = e^{2 \pi i f} e^{c_1(\cc O(D_{\rho}))} = \ch (\cc O(D_{\rho}))|_v.\]
This shows that the combinatorial Chern character of \cite{BH} coincides with the orbifold Chern character, and thus $\ch (D(\PP(G)))$ generates $H^*_{\op{CR}}(\PP(G))$.

\end{proof}

\begin{remark}
Recall $\tilde G = \bar G \times \CC^* \supset \bar G \times \langle j \rangle = G$.   
Characters of $\tilde G$ yield line bundles on $\PP(G)$.  Let  $\cc O_{ \PP(G)}( k + \rho_i) := \cc O_{ \PP(G)}( k) \otimes \cc  O_{ \PP(G)}( \rho_i)$ where $\cc O_{ \PP(G)}( k)$ corresponds to a character of weight $k$ on $\CC^*$ and  $\cc  O_{ \PP(G)}( \rho_i)$ corresponds to a character $\rho_i$ of $\hat G$.
By \cite[Theorem~4.10]{BH}, the relation
\[ \prod_{j=1}^N ( \cc O(c_j) - 1) = 0\]
holds in $K(\PP(c_1, \ldots, c_N))$, thus $H^*_{\op{CR}}(\PP(c_1, \ldots, c_N))$ is generated by $\ch( \{\cc O(k)\}_{k=0}^{\sum c_j -1})$.  From this one can check that $H^*_{\op{CR}}( \PP(G))$ is generated by $\ch( \{\cc O(k + \rho_i)\}_{0 \leq k< \sum c_j, \rho_i \in \widehat{\bar G}})$.
\end{remark}
%

\section{Local GW/FJRW correspondence and  quantum Serre duality}\label{s:GWFJRW}

\subsection{Local GW/FJRW correspondence}

%
%
\begin{definition}
Define the narrow state space transformation to be the linear map defined on the generating set by
\[\begin{array}{rrcl}
\Delta_-: & H^*_{\op{CR}, \op{nar}}(\cY_-)
&\to & 
\cc H_{\op{nar}}(w, G) \\
& \ii_g &\mapsto & \phi_{gj^{-1}}.
\end{array}
\] 
Define \[\bar \Delta_-:= (2 \pi i z)^{ \sum q_j} \cdot \Delta_-:  H^*_{\op{CR}, \op{nar}}(\cY_-)[z, z^{-1}] \to\cc H_{\op{nar}}(w, G)[z, z^{-1}].\]
\end{definition}

\begin{lemma}

Consider the functor $i^1_* \circ \pi_*$ where $\pi_*: \dabsfact{\cY_-, 0}_{BG} \to \dabsfact{BG, 0}$ and $i^1_*: \dabsfact{BG, 0} \to \dabsfact{\cY_-, w}$.  The induced functor on the narrow state space
is given by $\Delta_-$.
In particular, 
\begin{equation}\label{e:deltainduced} 
\Delta_- \circ \op{ch} = \op{ch} \circ i^1_* \circ \pi_* .\end{equation}
\end{lemma}

\begin{proof}
By Lemma~\ref{l:chgen} it suffices to show that 
\begin{equation}\label{e:deltacherno}\Delta_- \circ \op{ch} \circ i^0_*= \op{ch} \circ i^1_* \circ \pi_*\circ i^0_* = \op{ch} \circ i^1_*.\end{equation}
By \eqref{e:chzerobeta} and \eqref{e:chalphabeta}, we see directly that 
\[\Delta_- \circ \op{ch} \circ i^0_*(\cc O_{BG}) = \Delta_- \circ \op{ch} (\{0, \beta\}) = \op{ch} ( \{\alpha, \beta\})
= \op{ch} \circ i^1_* (\cc O_{BG}).\]
Another direct calculation as in \eqref{e:chzerobeta} and \eqref{e:chalphabeta} yields
\[\Delta_- \circ \op{ch} (V_\xi \otimes \{0, \beta\}) = \op{ch} ( V_\xi \otimes \{\alpha, \beta\})\]
for any character $\xi$ of $G$.
Since $G$ is abelian, it follows immediately that
\eqref{e:deltacherno}  holds on any $G$-representation $V_{\rho}$, and so holds on all of  $\dabsfact{BG}$.

%
%
\end{proof}

\begin{proposition}\label{p:Lcom-}
The map $\Delta_-$ commutes with $L$ after a change of variables:
\[ \Delta_- \circ L^{\cY_-}(\bt, z)  =  L^{(w, G)}(f(\bt), z) \circ \Delta_-,\]
where
\begin{equation}\label{e:f} f(\bt) = \Delta_- \left( \sum_{h \in nar} \langle \langle \ii_j, \ii_h \rangle \rangle^{\cY_-}(\bt) \ii_{h^{-1}} \right).\end{equation}
\end{proposition}

\begin{proof}
This result follows from the \emph{MLK correspondence} of \cite{LPS}: 
by Theorem~5.5 of \cite{LPS}, $\Delta_-$ identifies the  Lagrangian cone $\cc L^{0, \bs}$ with $\cc L^{1, \bs}$.  Where $\cc L^{c, \bs}$ is the (equivariant) $\bs$-twisted Lagrangian cone of Definition~\ref{cone}.  

 Since $z \partial_{\phi_j} J^{0, \bs}(\bt, -z)$ lies on $\sL^{0, \bs}$, \eqref{e:Jg0} and  \eqref{e:Jg} imply that
$\Delta_- (z \partial_{\phi_j} J^{0, \bs}(\bt, -z))$ is a $\CC[z]$-linear combination of derivatives of $J^{1, \bs}(\hat \bt, -z)$ at some point $\hat \bt$.  Analyzing the coefficients of non-negative powers of z in $\Delta_- (\partial_{\phi_j} J^{0, \bs}(\bt, -z))$, we see that $\Delta_- (z \partial_{\phi_j} J^{0, \bs}(\bt, -z))$ must equal $J^{1, \bs}(\hat f(\bt), -z)$, where
\[\hat f(\bt) = \Delta_- \left( \sum_{h} \langle \langle \ii_j, \ii_h \rangle \rangle^{0, \bs}(\bt) \ii_{h^{-1}} \right).\]  Therefore
\[\Delta_- ( \cc T_{J^{0, \bs}(\bt, -z)} \cc L^{0, \bs}) = \cc T_{J^{1, \bs}(\hat f(\bt), -z)} \cc L^{1, \bs},\]
and so by \eqref{e:Jtan} $\Delta_-(\partial_{\phi_g}J^{0, \bs}(\bt, -z))$ is a $\CC[z]$-linear combination of derivatives of $ J^{1, \bs}(\hat \bt, -z)$ evaluated at $\hat \bt = \hat f(\bt)$. Comparing $z^0$-coefficients, we see
\begin{equation}\label{e:deltaJtwisted}
\Delta_- (\partial_{\phi_g} J^{0, \bs}(\bt, -z)) = \partial_{\phi_{g j^{-1}}} J^{1, \bs}(\hat \bt, -z)|_{\hat \bt = \hat f(\bt)}.\end{equation}

We specialize \eqref{e:deltaJtwisted} to $\bs$ as in \eqref{e:spec}.  By Remark~\ref{r:twistedGW} the left hand side becomes the expression $\Delta_- (\partial_{\ii_g} J^{\cY_-, e^{-1}_{\CC^*}}(\bt, -z))$.
Note that every Gromov--Witten invariant appearing in the sum contains an insertion of $\ii_g$.  By Lemma~\ref{l:reltotwisted} there is a well-defined non-equivariant limit, supported on the span of the narrow sectors.  The left hand side becomes
\[\Delta_- (\partial_{\ii_g} J^{\cY_-}(\bt, -z)).\] Note finally that the change of variables $\hat \bt = \hat f(\bt)$ also has a well-defined non-equivariant limit to $\bt' = f(\bt)$.

Applying Lemma~\ref{l:specFJRW} to the $(1, \bs)$-twisted invariants, we see that the right hand side specializes to $\partial_{\phi_{g j^{-1}}} J^{(w, G)}(\bt', -z)|_{\bt' = f(\bt)}$.
We arrive at
%
%
%
%
%
\begin{equation}\label{e:deltaJ}
\Delta_- (\partial_{\ii_g} J^{\cY_-}(\bt, -z)) = \partial_{\phi_{g j^{-1}}} J^{(w, G)}(\bt', -z)|_{\bt' = f(\bt)}.\end{equation}

By 
\eqref{e:SL},
$L^\square(\bt, z)^{-1} \alpha = L^\square(\bt, -z)^T \alpha = \partial_\alpha J^\square(\bt, z)$.
 Equation \eqref{e:deltaJ} can be re-written as 
\begin{align*}\Delta_- \left((L^{\cY_-}(\bt, z))^{-1} (\ii_g)\right) &= (L^{(w, G)}(f(\bt), z))^{-1} (\phi_{g j^{-1}}) 
\\ &= (L^{(w, G)}(f(\bt), z))^{-1} (\Delta_- (\ii_g)).
\end{align*}
The conclusion follows.
\end{proof}

\begin{theorem}\label{t:MLK}
The map $\bar \Delta_-$ identifies the quantum $D$-module $QDM_{\op{nar}}(\cY_-)$ with $f^* \left( QDM_{\op{nar}}(w, G) \right)$ (up to multiplication by $e^{\pi i \sum q_j}$ in the pairing).  Furthermore it is compatible with the integral structure and the functor $i^1_* \circ \pi_*$, i.e., the following diagram commutes;
\begin{equation}\label{e:sident}
\begin{tikzcd}
D(\cY_-)_{BG} \ar[r, " i^1_* \circ \pi_*"] \ar[d, "s^{\cY_-, \op{nar}}"] &  i^1_*(D(BG)) \ar[d, "s^{w, G, nar}"]\\
QDM_{\op{nar}}(\cY_-) \ar[r, "\bar \Delta_-"] & QDM_{\op{nar}}(w, G).
\end{tikzcd}
\end{equation}
\end{theorem}

\begin{proof}
A simple check shows that $\langle \Delta_-(\alpha), \Delta_-(\beta) \rangle^{(w, G)}_{\op{nar}} = \langle \alpha, \beta \rangle^{\cY_-}_{\op{nar}}$.  It follows immediately that 
\[ S^{(w, G)}_{\op{nar}}(\bar \Delta_-(\alpha), \bar \Delta_-(\beta)) = e^{\pi i \sum q_j} S^{\cY_-}_{\op{nar}}(\alpha, \beta).\]
Because $s^\square$ generates the narrow quantum $D$-module for $\square$ equal to $\cY_-$ and $(w,G)$ respectively, the full identification of quantum $D$-modules will follow from \eqref{e:sident}.  
By Proposition~\ref{p:Lcom-} and \eqref{e:deltainduced}, we see that 
\[\Delta_- \circ s^{\cY_-, \op{nar}}(\bt, z) (E) = \frac{1}{(2 \pi i z)^{\sum q_j}} s^{(w, G), \op{nar}}(f(\bt),z) \circ i^1_* \circ \pi_*( E),\]
where the factor of $2 \pi i z$ arises from the difference in $\hat c $ and $\op{Gr}$ between the respective theories.  The conclusion follows.
\end{proof}

\subsection{Quantum Serre duality}
In this section we state an analogous result as in the previous section, but replacing $\cY_-$ by $\cY_+$.

Quantum Serre duality is a phenomenon taking many forms.  Originally proposed by Givental in \cite{G1} and later generalized in \cite{CG}, it relates different twisted invariants (as defined in Section~\ref{ss:twisty2}).  Two of the most crucial applications of twisted invariants are computing local Gromov--Witten invariants of the total space of a vector bundle in terms of the base and in calculating invariants of complete intersections in terms of the ambient space.  A particular case of the phenomenon relates these two invariants.  In \cite{IMM}, Iritani--Mann--Mignon reframed  quantum Serre duality in terms of quantum D-modules and the correspondence was shown to be compatible with Iritani's integral structures in a precise sense. 

In this section we state a  variation on the $D$-module interpretation of quantum Serre duality from \cite{IMM}, and show how this form of the correspondence is induced by a derived functor analogous to the previous section.  For more details on this formulation of quantum Serre duality in a general context, see \cite{Sh1}.

First we describe the state space isomorphism.  
Assume in what follows that $\cc O_{\PP(G)}(d)$ is convex.
Recall the following diagram.
\[
\begin{tikzcd}
 & \cY_+ \ar[d, "\pi"] \\
\cZ \ar[r, "j"] & \PP(G). 
\end{tikzcd}
\]
There is a well-defined pushforward map $\pi^c_*: H^*_{\op{CR}, \op{c}}(\cY_+) \to H^*_{\op{CR}}(\PP(G))$, therefore, for $\alpha \in H^*_{\op{CR}, \op{nar}}(\cY_+)$, we may obtain a pushforward to $H^*_{\op{CR}}(\PP(G))$ by  lifting $\alpha$ to $\tilde \alpha \in H^*_{\op{CR}, \op{c}}(\cY_+)$ and pushing forward via $\pi^c_*$.  The class $\pi^c_*(\tilde \alpha)$ is well-defined up to an element of the form $\pi^c_*(\kappa)$ where \[\kappa \in \ker(\phi: H^*_{\op{CR}, \op{c}}(\cY_+) \to H^*_{\op{CR}}(\cY_+)).\]  By Lemma 6.10 of \cite{Sh1}, $j^*(\pi^c_*(\kappa)) = 0$ for all $\kappa \in \ker(\phi)$.  Therefore, for elements $\alpha \in H^*_{\op{CR}, \op{nar}}(\cY_+)$, the composition $j^* \circ \pi_*$ is well-defined.
\begin{definition} Define the transformation $\Delta_+: H^*_{\op{CR}, \op{nar}}(\cY_+) \to H^*_{\op{amb}}(\cZ)$ 
to be the composition \[\Delta_+ := j^* \circ \pi_*.\]
Define \[\bar \Delta_+ := 2 \pi i z \Delta_+: H^*_{\op{CR}, \op{nar}}(\cY_+)[z, z^{-1}] \to H^*_{\op{amb}}(\cZ)[z, z^{-1}].\]
\end{definition}

An equivariant basis for the Chen--Ruan cohomology of $I\PP(G)$ is given by 
\[ \bigcup_{g \in G} \left\{ \tilde \ii_g, \tilde \ii_g H, \ldots, \tilde \ii_g H^{(\dim((\CC^N)^g) - 1)}\right\},
\]
where $\tilde \ii_g$ is the fundamental class of $\PP(G)_g$ and $\tilde \ii_g H^k$ denotes the pullback of the $k$th power of the hyperplane class from the course space of $\PP(G)_g$.  Here we use the convention that $\ii_g$ is zero if $\PP(G)_g$ is empty (i.e., if the action of $g$ on $\CC^N$ fixes only the origin).
Identifying these classes with their pullback to $I\cY_+$ via $\pi^*$ yields a basis for $H^*_{\op{CR}}(\cY_+)$.  
With this basis, the map $\Delta_+$ may be written as
\begin{align*}
\Delta_+: H^*_{\op{CR}, \op{nar}}(\cY_+) &\to H^*_{\op{CR}}(\cZ) \\
\sum_{g \in G, k \geq 1} c^g_k \tilde \ii_g H^k &\mapsto 
 -\frac{1}{d} \sum_{g \in G, k \geq 1} c^g_k j^*(\tilde \ii_g H^{k-1}).
\end{align*}
Note that classes of the form $j^*(\tilde \ii_g H^{k-1})$ span $H^*_{\op{CR}, \op{amb}}(\cZ)$, with \[j^*(\tilde \ii_g H^{(\dim((\CC^N)^g) - 1)}) = 0\] for each $g \in G$.  
Under assumption~\ref{a:ass2}, one can check that $\Delta_+$ is an isomorphism \cite[Lemma~6.11]{Sh1}.

We will show that $\bar \Delta_+$ gives an isomorphism of quantum $D$-modules and is compatible with integral structure and the functor 
\[j^* \circ \pi_*: D(\cY_+)_{\PP(G)} \to D(\cZ).\]
We have the following orbifold Grothendieck--Riemann--Roch statement.  
\begin{lemma}\cite[Lemma~6.4]{Sh1}\label{l:ch+}
Consider the functor \[j^* \circ \pi_*: D(\cY_+)_{\PP(G)} \to D(\cZ)\] and the map on narrow state spaces 
\[\Delta_+ (\op{Td}(\cc O_{\cY_+}(-d)) \cup - ): H^*_{\op{CR}, \op{nar}}(\cY_+) \to H^*_{\op{CR}, \op{amb}}(\cZ).\]
These maps are related via the Chern character, i.e.
\begin{equation}\label{e:chcomm}\Delta_+ \left(\op{Td}(\cc O_{\cY_+}(-d)) \cup  \op{ch} (F)\right)= \op{ch} \circ j^* \circ \pi_*(F)\end{equation}
for all $F \in D(\cY_+)_{\PP(G)}$.
\end{lemma}

The following is a special case of Proposition 6.13 in \cite{Sh1}. 
\begin{proposition}\cite[Proposition 6.13]{Sh1}\label{p:Lcom+}
The following operators are equal after a change of variables: 
\[  \Delta_+ \circ L^{\cY_+}(\bt, z) =  L^{\cZ}(\bar  f_+(\bt), z)  \circ \Delta_+\circ e^{-\pi i dH/z},\]
where 
\begin{equation}\label{e:barf} \bar f_+(\bt) = \Delta_+ \left( \sum_{i \in I_{\op{nar}}} \langle \langle -dH, T_i \rangle \rangle^{\cY_+}(\bt ) T^i \right) - \pi i dH.\end{equation}
\end{proposition}
\begin{remark}
A very similar statement was shown in the non-orbifold case in Theorem 3.14 in \cite{IMM}.  Note that the statements differ in both the direction of the map $\Delta_+$ and in the change of variables, one is the inverse of the other.  See Remark 6.6 of \cite{Sh1} for more details.
\end{remark}

The following is a special case of Theorem 6.14 of \cite{Sh1}.  It follows from Proposition~\ref{p:Lcom+} and Lemma~\ref{l:ch+}.
\begin{theorem}\cite[Theorem~6.14]{Sh1}\label{t:QSD}
The map $\bar \Delta_+$ gives an isomorphism of quantum $D$-modules between $QDM_{\op{nar}}(\cY_+)$ and $ \bar f^*_+ \left( QDM_{\cZ, \op{amb}} \right)$.  Furthermore it is compatible with the integral structure and the functor $j^* \circ \pi_*$, i.e., the following diagram commutes;
\[
\begin{tikzcd}
D(\cY_+)_{\PP(G)} \ar[r, " j^* \circ \pi_*"] \ar[d, "s^{\cY_+, \op{nar}}"] &  j^*(D(\PP(G))) \ar[d, "s^{\cZ, \op{amb}}"]\\
QDM_{\op{nar}}(\cY_+) \ar[r, "\bar \Delta_+"] & QDM_{\op{amb}}(\cZ).
\end{tikzcd}
\]
%
\end{theorem}

\section{Relation to CTC and LG/CY correspondence}\label{s:CTC}

The crepant transformation conjecture relates the Gromov--Witten theory of $K$-equivalent smooth Deligne--Mumford stacks $\cX_+$ and $\cX_-$.  Although there are many formulations, as it is stated in e.g. \cite{CIJ}, it relates the quantum connections $QDM(\cX_+)$ and $QDM(\cX_-)$.  In \cite{CIJ} the conjecture is proven for toric varieties.  In the non-compact case this is defined via localization of equivariant cohomology. 

The LG/CY correspondence takes a similar form, but relates the FJRW theory of $(w,G)$ with that of $\cZ$.
In \cite{LPS}, we show that a form of the LG/CY correspondence is implied by the crepant transformation conjecture.  
In particular we show via the local GW/FJRW correspondence and quantum Serre duality that the transformation $\cc U$ 
relating the Gromov--Witten theory of $\cY_-$ and $\cY_+$ can be used to define a transformation relating $(w,G)$ 
and $\cZ$.  The proof involves taking the non-equivariant limit of a restriction of the correspondence in \cite{CIJ}.

In this section we formulate a version of the crepant transformation conjecture for non-compact spaces which does not require the use of equivariant cohomology.  This formulation together with the results of the previous two sections directly implies the LG/CY correspondence, and includes the connection between integral structures and Orlov's equivalence as in \cite{CIR}.  The results of this section are a refinement of the work in \cite{LPS} to incorporate  Iritani's integral structures and the derived equivalences discussed in Section~\ref{s:structures}.

As in Section~\ref{ss:Dequiv}, we assume for this section that $\sum_{j=1}^N c_j = d$.  We further assume that $\cc O_{\PP(G)}(d)$ is convex, which adds the additional constraint that $\cc O_{\PP(G)}(d)$ is pulled back from the coarse space of $\PP(G)$ (see Remark~5.3 of \cite{CGIJJM}).  This in turn implies that
$G \leq SL_N(\CC)$.

\subsection{Crepant transformation conjectures}

\subsubsection{Torus-equivariant CTC via $I$-functions}
Let $T \cong \CC^*$ act on the coordinates of $\cY_- = [\CC^N/G]$ with weights $-c_1, \ldots, -c_N$.  On $\cY_+ = \text{Tot}(\cc O_{\PP(G)}(-d))$ this corresponds to a trivial action on the base $\PP(G)$ with a nontrivial action of weight $-d$ in the fiber direction.\footnote{This choice of action on $\cY_-$ is used also in \cite{LPS}, where it is erroneously stated that the corresponding weight in the fiber direction of $\cY_+$ is $d$.}   We let $\lambda$ denote the equivariant parameter of our torus action in $H^*_{CR,T}(\cY_-)$ and $H^*_{CR,T}(\cY_+)$.  Let $R_T$ denote $H^*(BT) = \CC[\lambda]$.

Let $S \subset G$ denote the subset $G$ consisting of elements which fix at least one coordinate of $\CC^N$.
For each $g \in S$ let $t^g$ denote a corresponding formal variable.   Let $t$ denote a coordinate corresponding to $j$.  We will consider the ring $\CC[[t, \bt]] := \CC[[t]] \otimes \CC[[t^g]]_{g \in S}$.
%

Let $a(\bv{k})^j=\sum_{g \in S}k_g m_j(g)$.  Define the modification factor
\[
M(k_0,\bv{k}) := \prod_{j=1}^N \prod_{l=0}^{\lfloor k_0q_j+a(\bv{k})^j \rfloor-1}\Big(-c_j\lambda-(\langle k_0 q_j+a(\bv{k})^j  \rangle +l)z\Big)
\]
where $\langle -\rangle$ denotes the fractional part. Then $I^{\cY_-}(t, \bt, z)$ is defined as 
\begin{equation}
I^{\cY_-}(t, \bt,z)= z t^{d\lambda/z}
\sum_{\bv{k}\in(\ZZ_{\geq 0})^{S}}\prod_{g \in S}\frac{(t^{g})^{k_g}}{z^{k_g}k_g!}\sum_{k_0\geq 0}\frac{M(k_0,\bv{k})t^{k_0}}{z^{k_0}k_0!}\ii_{\jj^{k_0}\prod_g g^{k_g}}.
\end{equation}
The above modification factor is explained in Section~4.2 of \cite{CCIT}, where it is proven that $I^{\cY_-}(t, \bt, z)$ is a 
$H^*_{CR, T}(\cY_-)[[t, \bt]]((z^{-1}))$-valued point on $\cc L^{\cY_-}$.

Recall our choice of equivariant basis for the Chen--Ruan cohomology of $\cY_+$ is given by 
\[ \bigcup_{g \in G} \left\{ \tilde \ii_g, \tilde \ii_g H, \ldots, \tilde \ii_g H^{(\dim((\CC^N)^g) - 1)}\right\}.
\]
Here we will use the convention that $\tilde \ii_g$ is zero if ${\cY_+}_g$ is empty (i.e., if $g \notin S$).
Let $t^g$ denote a formal coordinate corresponding to the fundamental class $\tilde \ii_g$ on ${\cY_+}_g$.  Let $q$ denote a coordinate corresponding to the
hyperplane class $H$. 

An $I$-function for toric stacks is given as Definition~28 of \cite{CCIT2}.  
Define
\begin{align}\nonumber
I^{\cY_+}(q,\bt,z) =&zq^{H/z} \sum_{\bv{k}\in(\ZZ_{\geq 0})^{S}}\prod_{g \in S}\frac{(t^{g})^{k_g}z^{(\age(g)-1)k_g}}{k_g!} \sum_{k_0\geq 0}\frac{q^{k_0/d}}{z^{k_0(\sum_j q_j - 1)+\sum_j\langle k_0q_j-a(\bv k)^j\rangle}} \\ \label{e:IY}
	&\cdot \frac{\Gamma(1-\tfrac{d(\lambda+H)}{z})}{\Gamma(1-k_0-\tfrac{d(\lambda+H)}{z})}\prod_{j=1}^N\frac{\Gamma(1+c_jH/z-\langle -k_0q_j+a(\bv{k})^j \rangle)}{\Gamma(1+c_jH/z+k_0q_j-a(\bv{k})^j)}\tilde \ii_{\jj^{-k_0}\prod_g g^{k_g}}
\end{align}
Similar to the above,
$I^{\cY_+}(q, \bt, z)$ is a 
$H^*_{CR, T}(\cY_+)[[q, \bt]][\log q]((z^{-1}))$-valued point on $\cc L^{\cY_+}$.


\begin{definition}\label{d:d}
 Denote $U_+$ to be the universal covering of $\{q: q \neq c\}$, where 
\[c = (-d)^{d}\prod_{i = 1}^N c_i^{-c_i}.\]
Denote by  $U_-$ the universal cover of $\{t: t^d \neq c^{-1}\}$.  Finally, let $U$ denote the universal cover of 
$\{t: t^d \neq c^{-1}, t \neq 0\}$.  There is a natural induced map 
\[ \pi_-: U \to U_-.\]
Via the change of variables $q = t^{-d}$, the open set $\{t: t^d \neq c^{-1}, t \neq 0\}$ maps to $\{q: q \neq c\}$.  This induces another map
\[\pi_+: U \to U_+.\]
\end{definition}
A ratio test shows that 
$I^{\cY_+}(q, \bt, z)$ is convergent in the $q$ direction,
and may be analytically continued to $U_+$.
Similarly
$I^{\cY_-}(t, \bt, z)$ is convergent in the $t$ direction,  and may be analytically continued to $U_-$.  Thus both $I$-functions may be viewed as functions on $U$.
Under the change of variables $q = t^{-d}$, we may analytically continue $I^{\cY_+}$ from $\log q = -\log t^d = - \infty$ to $\log q = - \log t^d = \infty$ along any path avoiding $\log q = \log |c| + d \pi i + 2 \pi i \ZZ$.  Let $\gamma_l$ denote such a path, which moves from $(im(\log q) = 0, re(\log q) = -\infty)$ to $(im(\log q) = 0, re(\log q) = \infty)$ with $re(\log q)$ always increasing, and such that when $re(\log q) = \log |c|$, $im(\log q) = (d + 2l  - 1) \pi i$.

\begin{definition}\label{d:utl}
For each $l \in \ZZ$, define the map
\begin{align*}
\overline \UU^T_l\colon &H^*_{CR, T}(\cY_-) \to H^*_{CR, T}(\cY_+) \\
&\ii_{g j^m} \mapsto \frac{1}{d}\sum_{b=0}^{d-1}\frac{\left( \xi^{b+m} e^{(H + \lambda)}\right)^l\left(e^{d(H + \lambda)}-1\right)}{e^{(H+\lambda)} \xi^{b +m} - 1} \tilde \ii_{gj^{-b}},
\end{align*}
where $\xi  = e^{2 \pi i/d}$.
\end{definition}
\begin{lemma}\label{l:induced}
The following diagram commutes:
\[
\begin{tikzcd}
D_T(\cY_-) \ar[r, "\vgit^T"] \ar[d, "I^*\circ \op{ch}"] &  D_T(\cY_+) \ar[d, "I^*\circ \op{ch}"]\\
H^*_{CR, T}(\cY_-)\ar[r, " \overline \UU^T_l"] & H^*_{CR, T}(\cY_+)
\end{tikzcd}
\]
where $\vgit^T$ is the $T$-equivariant version of the equivalence in Section~\ref{ss:Dequiv}.
\end{lemma}

\begin{proof}
We first note that it suffices to check commutativity on elements of $D_T(\cY_-)$ represented by line bundles.   A $T$-equivariant vector bundle $\cV$ on $\cY_-$  is equivalent to a $T \times G$-equivariant vector bundle on $\CC^N$.  Because all vector bundles on $\CC^N$ are trivial and $T \times G$ is abelian, $\cV$ will split into a sum of line bundles. Any element of $D_T(\cY_-)$ has a finite resolution by $T$-equivariant vector bundles.

For notational simplicity we will focus on the case $ G = \langle j \rangle$, so $\tilde G = \CC^*$.  The  general case is similar.  The $T \times \CC^*$-character of weight $(0,k)$ yields a $T$-equivariant line bundle on $[V/\CC^*]$.  This pulls back to $\CC_{\br{k/d}}$ on $\cY_-$, the line bundle such that the $\mu_d$ isotropy at the origin acts with weight $k$ and the torus acts trivially.  On $\cY_+$, 
this pulls back to $\cc O_{\cY_+}(k) \otimes \cc O_T(k)$ where $\cc O_{\cY_+}(k)$ is pulled back from $\PP(G)$ and has trivial $T$-action, and $\cc O_T(k)$ is the trivial line bundle with a $T$-action of weight $k$.  
It follows that for $k$ in the range $[l , l+d -1]$, $\vgit^T (\CC_{\br{k/d}}) = \cc O_{\cY_+}(k) \otimes \cc O_T(k)$, and thus for the diagram to commute, $\overline \UU^T_l$ must map 
\[I^* \circ \ch (\CC_{\br{k/d}}) = \sum_{0 \leq m < d} \xi^{-mk}
\ii_{j^m}\]
 to 
\[I^* \circ \ch (\cc O_{\cY_+}(k) \otimes \cc O_T(k)) = \sum_{0 \leq b < d} \xi^{bk} 
e^{k(H + \lambda)} \tilde \ii_{j^{-b}}.\]
This is equivalent to the requirement that
\begin{align*} \ii_{j^m} &= \frac{1}{d} \sum_{k = l}^{l+d-1} \xi^{km} 
I^* \circ \ch (\CC_{\br{k/d}})  \\
&=\frac{\xi^{lm}}{d} \sum_{k = 0}^{d-1} \xi^{km} 
I^* \circ \ch (\CC_{\br{(l + k)/d}})
\end{align*}
maps to 
\begin{align*} 
&\frac{\xi^{lm}}{d} \sum_{k = 0}^{d-1} \xi^{km} 
I^* \circ \ch (\cc O_{\cY_+}(l+ k) \otimes \cc O_T(l+k)) \\
 = &\sum_{b = 0}^{d-1} \frac{\xi^{l(b+m)}e^{l(H+ \lambda)}}{d} \sum_{k=0}^{d-1} \xi^{k(b+m)} e^{k(H + \lambda)} \tilde \ii_{j^{-b}}.
\end{align*}
Using the simple fact that $\sum_{k=0}^{d-1} x^k = (x^d - 1)/(x-1)$, the above becomes
\begin{align*} 
 &\sum_{b = 0}^{d-1} \frac{\xi^{l(b+m)}e^{l(H+ \lambda)}}{d} \sum_{k=0}^{d-1} 
 \frac{(\xi^{b+m}e^{(H + \lambda)})^d - 1}{\xi^{b+m}e^{(H + \lambda)} - 1}\tilde \ii_{j^{-b}}
\\
 =& \frac{1}{d}\sum_{b=0}^{d-1}\frac{\left( \xi^{b+m} e^{(H + \lambda)}\right)^l\left(e^{d(H + \lambda)}-1\right)}{e^{(H+\lambda)} \xi^{b +m} - 1} \tilde \ii_{j^{-b}},
\end{align*}
in agreement with Definition~\ref{d:utl}.
\end{proof}

%
%
%

\begin{theorem}\label{t:mI}
For each path $\{\gamma_l\}_{l \in \ZZ}$, there exists a linear transformation 
\[\UU^T_l: H^*_{CR, T}(\cY_-)((z^{-1})) \to H^*_{CR, T}(\cY_+)((z^{-1}))\] satisfying
\begin{enumerate}
\item $\UU^T_l( I^{\cY_-}(t, \bt, z)) = \widetilde I^{\cY_+}(t, \bt, z)$ where $ \widetilde I^{\cY_+}(t, \bt, z)$ denotes the analytic continuation of $I^{\cY_+}$ along the path $\gamma_l$ to a neighborhood of $t = 0$.
\item  $\UU^T_l$ is induced by the equivalence $\vgit^T$ in the sense that the following diagram commutes:
\[
\begin{tikzcd}
D_T(\cY_-) \ar[r, "\vgit^T"] \ar[d, "\psi^{\cY_-}"] &  D_T(\cY_+) \ar[d, "\psi^{\cY_+}"]\\
H^*_{CR, T}(\cY_-)((z^{-1})) \ar[r, "\UU^T_l"] & H^*_{CR, T}(\cY_+)((z^{-1}))
\end{tikzcd}
\]
where $\psi^\square = z^{- \op{Gr}}z^\rho \hat \Gamma_\square (2 \pi i)^{\deg_0 /2} I^* \circ \op{ch}$.
\item The map $\UU^T_l$ preserves the symplectic pairing.
\end{enumerate}
\end{theorem}

\begin{proof}  The computation of $\UU^T_l$ is now standard so we only sketch the details.
The analytic continuation of $I^{\cY_+}$ was computed in the appendix of \cite{LPS} for the path $\gamma_0$.  The computation can be easily modified for the path $\gamma_l$ by multiplying the integrand of the line integral on page 1436 of \cite{LPS} by a factor of $e^{-2 \pi i l s}$.  Following \cite{LPS}, if we define
 \[\UU^T_l :=z^{-\Gr}\hat{\Gamma}(\cY_+)(2\pi i)^{\deg_0/2}\overline \UU^T_l (2\pi i)^{-\deg_0/2}\hat{\Gamma}(\cY_-)^{-1}z^{\Gr},\]
then $I^{\cY_-}(t, \bt, z) $ will map to $ \widetilde I^{\cY_+}(t, \bt, z)$.  This proves the first claim.  The second claim follows by Lemma~\ref{l:induced} above.
The third point follows from the second, together with the fact that the maps $\psi^\square$ identify the Euler pairing with the symplectic pairing, and $\vgit^T$ preserves the Euler pairing.

\end{proof}

This was proven in \cite[Theorem~6.1]{CIJ} in the case where $l = 0$ and $\op{vGIT}_0^T$ was replaced by a particular Fourier--Mukai transform $\mathbb{FM}$.  In \cite[Proposition~5.3]{CIJS} it was shown that $\op{vGIT}_0^T = \mathbb{FM}$.  

\begin{remark}
The choice of coordinates of the $I$-functions -- indexed by $g \in S$ and $j$ -- may seem somewhat mysterious.  Indeed there are other choices of $I$-functions $I^{\cY_-}$ and $I^{\cY_+}$ with coefficients in different rings.  Each choice corresponds to a different presentation of $\cY_-$ and $\cY_+$ as a toric GIT quotient; see the work of Jiang  on extended stacky fans \cite{Jia}.  For each such choice, an analogous result to Theorem~\ref{t:mI} will hold, as shown in \cite[Theorem~6.1]{CIJ}.  Because our purpose in this section is not to reprove the crepant transformation conjecture but rather to show how it leads to the LG/CY correspondence, we will content ourselves with working in the context of the $I$-functions defined above, where the analytic continuation and formula for $\UU^T_l$ have already been computed explicitly.  
\end{remark}

\subsubsection{Torus-equivariant CTC via $D$-modules}

The mirror theorem for a toric variety such as $\cY_-$ or $\cY_+$ can be rephrased in terms of equivariant quantum $D$-modules.  To do this one must first define an equivariant version of the quantum connection and integral structures of Section~\ref{s:structures}.  

We may extend the  quantum product $\bullet_{\bt}^{\cY_{+/-}}$  to an equivariant quantum product via:
\[ \langle \alpha \bullet_{\bt}^{\cY_{+/-}^T} \beta, \gamma \rangle^{\cY_{+/-}^T} = \langle \langle \alpha, \beta, \gamma\rangle \rangle^{\cY_{+/-}^T}(\bt)\]
where $\cY_{+/-}^T$ is used to denote the equivariant theory (i.e. the equivariant pairing and Gromov--Witten invariants).
We then define 
\[\nabla_i^{\cY_{+/-}^T} = \nabla_{\partial/\partial t^i}{\cY_{+/-}^T} := \partial_i + \frac{1}{z} \bullet_{\bt}^{\cY_{+/-}^T}.\]
The equivariant Euler vector field in these cases is given by 
\[ \cc E^T := - d \lambda \frac{\partial}{\partial \lambda} +  \partial_\rho+ \sum_{i \in I} \left(1 - \frac{1}{2} \deg T_i\right) t^i \partial_i,\]
and the grading operator extends by defining
$\op{Gr}(\lambda) = 1$.  Then define
\[\nabla_{z}^{\cY_{+/-}^T} := \partial_z - \frac{1}{z^2} \cc E \bullet^{\cY_{+/-}^T}_{\bt} + \frac{1}{z} \op{Gr}.\]
The equivariant quantum solution $L^{\cY_{+/-}^T}$ is defined in the obvious way, and the equivariant flat sections $\bs^{\cY_{+/-}^T}$ are obtained by replacing $\ch$, $\Gamma_{\cY_{+/-}}$, and $L^{\cY_{+/-}}$ with their equivariant counterparts.

Recall Definition~\ref{d:d} of $U_+$, $U_-$ and $U$.
Let $M_{+}$ be a formal neighborhood of $U_{+}$ in $U_{+} \times \spec(\CC[\bt])$ (similarly for $M_-$ and $M$).  
We will make use of the following sheaf on $M_{+/-}$,
\[\cc O_{M_{+/-}} := \cc O^{\op{an}}_{U_{+/-}} [[ \bt]],\]
where $\cc O^{\op{an}}_{U_{+/-}}$ is the sheaf of analytic functions on $U_{+/-}$.  We will refer to this as the structure sheaf on $M_{+/-}$.

The following theorem is a special case of Theorem 5.15 of \cite{CIJ}, the argument for which appeared previously in Lemma 4.7 of~\cite{Iri}.  The proof 
follows from the fact that $I^{\cY_+}$ lies on the Lagrangian cone $\cc L_{\cY_+}$.

\begin{theorem}\cite[Theorem 5.15]{CIJ} \label{t:mtD}
There exists an open subset $U_+^\circ$ containing a neighborhood of (the preimage of) the origin such that $U_+ \setminus U_+^\circ$ is discrete, and if we let $M_+^\circ = M_+|_{U_+^\circ}$, we have the following over $M_+^\circ(R_T) \times \CC_z$:
\begin{itemize}
\item a flat connection $\nabla^+ = d + \mathbf{A}(q, \bt, z)$ defined over
\[\mathbf{F}^+ = H^*_{CR, T}(\cY_+) \otimes_{R_T} \cc O_{M_+^\circ(R_T) \times \CC_z};\]
\item a mirror map $\tau_+^T: M_+^\circ(R_T) \to H^*_{CR, T}(\cY_+)$;
\item a global section $Y^+(q, \bt, z)$ of $\mathbf{F}^+$;
\end{itemize}
such that $\nabla^+$ is equal to the pullback $(\tau_+^T)^* \nabla^{\cY_+^T}$ of the quantum $D$-module of $\cY_+$, 
and $I^{\cY_+}(q, \bt, z) = z L^{-1}(\tau_+^T(q, \bt), z) Y^+(q, \bt, z)$ on $\{z \neq 0\}$.
\end{theorem}
In the above, $M_+^\circ(R_T)$ denotes the space $M_+^\circ$, but with the structure sheaf replaced by 
\[\cc O_{M_+^\circ} \otimes R_T,\]
and $\CC_z$ denotes $\CC$ with coordinate $z$.
 Theorem 5.15 of \cite{CIJ} shows that the matrix $\mathbf{A}(q, \bt, z)$ takes a particularly simple form, but we will not use this fact directly in what follows.  
The analogous statement holds for $\cY_-$ on $U_-$.

By the $D$-module mirror theorem described above for $\cY_+$ and $\cY_-,$ we have $D$-modules $(\mathbf{F}^{+/-}, \nabla^{+/-})$ over $M_{+/-}^\circ$.  Denote 
\[U^\circ := \pi_+^{-1}\big(U_+^\circ\big) \cap \pi_-^{-1}\big(U_-^\circ\big)\] and let $M^\circ$ 
denote the intersection
\[M^\circ :=  \pi_+^{-1}\big(M_+^\circ\big) \cap \pi_-^{-1}\big(M_-^\circ\big)\] over $U^\circ$.  By a slight abuse of notation, denote by $\pi_{+/-}$ the maps \[M^\circ(R_T) \to M^\circ_{+/-}(R_T).\]  Both $D$-modules on $M_+^\circ(R_T)$ and $M_-^\circ(R_T)$ pull back to $M^\circ(R_T)$.
Then a $D$-module formulation of the crepant transformation conjecture asserts that there exists (for each $l \in \ZZ$) a gauge transformation $\Theta_l$ over $M^\circ(R_T)$ relating the two $D$-modules.
The main theorem (Theorem 6.3) of \cite{CIJ}, in the special case of $\cY_-$ and $\cY_+$ gives the theorem for $l = 0$.  The same methods prove the theorem for general $l \in \ZZ$.
%

Define $\Theta_l^T\in \hom \left( H^*_{CR, T}( \cY_-), H^*_{CR,T}(\cY_+)\right) \otimes_{R_T} \cc O_{{M^\circ(R_T)} \times \CC_z}$ as 
\[
\Theta_l^T := L^{\cY_+^T}(\tau_+^T\circ \pi_+( q, \bt), z) \circ \UU_l \circ L^{\cY_-^T}(\tau_-^T\circ \pi_-(q, \bt), z)^{-1}.\]  \begin{remark} In the ``Proof that Theorem 6.1 implies 6.3'' of \cite{CIJ} it is shown that this matrix is in fact polynomial in $z$.  
\end{remark}
\begin{theorem}\cite[Theorem 6.3]{CIJ} \label{t:CTCD} The gauge transformation
$\Theta_l^T 
$ satisfies the following:
\begin{itemize}
\item $\nabla^-$ and $\nabla^+$ are gauge equivalent via $\Theta_l^T$, i.e.
\[ \nabla^+ \circ \Theta_l^T = \Theta_l^T \circ \nabla^-;\]
\item for all $E \in D^T(\cY_-)$, 
\[ \Theta_l^T (\bs^{\cY_-^T}(E)(\tau_-^T \circ \pi_-(t, \bt),z)) = \bs^{\cY_+^T}(\vgit^T(E))(\tau_+^T\circ \pi_+(q, \bt), z);\]
\item $\Theta_l^T$ preserves the pairing, i.e. $S^{\cY_+}( \Theta_l^T(\alpha), \Theta_l^T(\beta)) = S^{\cY_-}(\alpha, \beta)$.
\end{itemize}
\end{theorem}

\begin{proof}
From its definition, it is immediate that $\Theta_l^T$ sends the solution matrix for the pullback of $\nabla^{\cY_-^T}$ to the solution matrix for the pullback of $\nabla^{\cY_+^T}$.
The second point is immediate from part (2) of Theorem~\ref{t:mI} and the definition of the flat sections $\bs^\square$.
The third point follows from the second together with the fact from Proposition~\ref{p:pair} that $\bs^\square$ identifies the pairing with the Euler pairing in the derived category.
\end{proof}

\subsubsection{Narrow CTC}

In this section we use the previous theorems on the equivariant quantum $D$-modules of $\cY_-$ and $\cY_+$ to formulate statements involving the narrow quantum $D$-modules of Definition~\ref{d:nard}.  This formulation makes no mention of equivariant cohomology.

%

First we define narrow versions of the connections $\nabla^{+/-}$.
The connection $\nabla^{\cY_+^T}$ has a well defined non-equivariant limit $\nabla^{\cY_+}$ as described in part (3) of Section~\ref{ss:spec}.
The restriction of $\nabla^{\cY_+}$ to $H^*_{\op{CR}, \op{nar}}(\cY_+)$ yields the flat connection $\nabla^{\cY_+, \op{nar}}$.  Define \[\tau_+: M^\circ \to H^*_{\op{CR}}(\cY_+)\] as the non-equivariant limit of $\tau_+^T \circ \pi_+$, note that $\tau_+$ maps to all of $H^*_{\op{CR}}(\cY_+)$, not just the narrow subspace.
\begin{definition}  We define the connection $\nabla^{+, \op{nar}}$ on 
\begin{align*} \mathbf{F}^{+, \op{nar}} & := H^*_{\op{CR}, \op{nar}}(\cY_+) \otimes \cc O_{M^\circ \times \CC_z} = \tau_+^*(H^*_{\op{CR}, \op{nar}}(\cY_+)) \\
\nabla^{+, \op{nar}} &:= \tau_+^*(\nabla^{\cY_+, \op{nar}}).
\end{align*}
\end{definition}
Define $\mathbf{F}^{-, \op{nar}}$ and $\nabla^{-, \op{nar}}$ analogously.

\begin{definition}
Define $\overline \UU_l$ to be the non-equivariant limit of $\overline \UU_l^T$.
\begin{align*}
\overline \UU_l\colon &H^*_{\op{CR}}(\cY_-) \to H^*_{\op{CR}}(\cY_+) \\
&\ii_{g j^m} \mapsto \frac{1}{d}\sum_{b=0}^{d-1}\frac{\left( \xi^{b+m} e^{H }\right)^l\left(e^{dH}-1\right)}{e^{H} \xi^{b +m} - 1} \tilde \ii_{gj^{-b}},
\end{align*}
where $\xi  = e^{2 \pi i/d}$.  Observe that $\overline \UU_l$ maps $H^*_{\op{CR}, \op{nar}}(\cY_-)$ to $H^*_{\op{CR}, \op{nar}}(\cY_+)$.
Define \[\overline \UU_l^{\op{nar}} := \overline \UU_l |_{H^*_{\op{CR}, \op{nar}}(\cY_-)} : H^*_{\op{CR}, \op{nar}}(\cY_-)\to H^*_{\op{CR}, \op{nar}}(\cY_+).\]
\end{definition}
  
\begin{lemma}\label{l:narind}
The following diagram commutes:
\[
\begin{tikzcd}
D(\cY_-)_{BG} \ar[r, "\vgit"] \ar[d, "I^*\circ \op{ch}"] &  D(\cY_+)_{\PP(G)} \ar[d, "I^*\circ \op{ch}"]\\
H^*_{\op{CR}, \op{nar}}(\cY_-)\ar[r, " \overline \UU_l"] & H^*_{\op{CR}, \op{nar}}(\cY_+)
\end{tikzcd}
\]
\end{lemma}
\begin{proof}
By forgetting the $T$-action in Lemma~\ref{l:induced}, the following diagram is seen to commute:
\[
\begin{tikzcd}
D(\cY_-) \ar[r, "\vgit"] \ar[d, "I^*\circ \op{ch}"] &  D(\cY_+) \ar[d, "I^*\circ \op{ch}"]\\
H^*_{\op{CR}}(\cY_-)\ar[r, " \overline \UU_l"] & H^*_{\op{CR}}(\cY_+).
\end{tikzcd}
\]
From this the lemma follows immediately, after noting that $I^* \circ \ch$ maps $D(\cY_-)_{BG}$ (resp. $D(\cY_+)_{\PP(G)}$) onto $H^*_{\op{CR}, \op{nar}}(\cY_-)$ (resp. $H^*_{\op{CR}, \op{nar}}(\cY_+)$).
\end{proof}

With this we can state the narrow crepant transformation conjecture.
\begin{theorem}\label{t:CTCnar} There exists a gauge transformation
\[
\Theta_l^{\op{nar}} \in \hom \left( H^*_{\op{CR}, \op{nar}}( \cY_-), H^*_{\op{CR}, \op{nar}}(\cY_+)\right) \otimes \cc O_{{M^\circ} \times \CC_z}
\] such that:
\begin{itemize}
\item $\nabla^{-, \op{nar}}$ and $\nabla^{+, \op{nar}}$ are gauge equivalent via $\Theta_l$, i.e.
\[ \nabla^{+, \op{nar}} \circ \Theta_l^{\op{nar}} = \Theta_l^{\op{nar}} \circ \nabla^{-, \op{nar}};\]
\item for all $E \in D(\cY_-)_{BG}$, 
\[ \Theta_l^{\op{nar}} (\bs^{\cY_-, \op{nar}}(E)(\tau_-(t, \bt),z)) = \bs^{\cY_+, \op{nar}}(\vgit(E))(\tau_+(q, \bt), z);\]
\item $\Theta_l^{\op{nar}}$ preserves the pairing, i.e. $S^{\cY_+, \op{nar}}( \Theta_l^{\op{nar}}(\alpha), \Theta_l^{\op{nar}}(\beta)) = S^{\cY_-, \op{nar}}(\alpha, \beta)$.
\end{itemize}
\end{theorem}

\begin{proof}
The proof follows almost immediately from Theorem~\ref{t:CTCD} by taking non-equivariant limits.
Define $\Theta_l$ to be the non-equivariant limit of $\Theta_l^T$, and define $\Theta_l^{\op{nar}}$ to be its restriction to $H^*_{\op{CR}, \op{nar}}( \cY_-)$.  Note that by Proposition~\ref{p:narpres} both $L^{\cY_+}$ and $L^{\cY_-}$ (and their inverses) preserve the narrow state spaces.  Therefore 
\[ \Theta_l|_{H^*_{\op{CR}, \op{nar}}( \cY_-)} =  L^{\cY_+}|_{H^*_{\op{CR}, \op{nar}}( \cY_+)} \circ (\UU_l)|_{H^*_{\op{CR}, \op{nar}}( \cY_-)}\circ (L^{\cY_-})^{-1} |_{H^*_{\op{CR}, \op{nar}}( \cY_-)}.\]

It is then immediate that $\Theta_l^{\op{nar}}$ sends the solution matrix for the pullback of $\nabla^{\cY_-, \op{nar}}$ to the solution matrix for the pullback of $\nabla^{\cY_+, \op{nar}}$.
The second point follows from Lemma~\ref{l:narind} and the definition of the flat sections $\bs^\square$.
The third point follows from the second together with the fact from Proposition~\ref{p:pair} that $\bs^\square$ identifies the pairing with the Euler pairing in the derived category.
\end{proof}

%
%
%
%
%
%
%
%
%
%

\subsection{The LG/CY correspondence via CTC}\label{ss:lgcyctc}
In this section we show how the LG/CY correspondence follows directly from the narrow crepant transformation conjecture above, together with Theorems~\ref{t:MLK} and~\ref{t:QSD}.

\begin{definition}
Define $\bar \tau_+: M^\circ_+ \to H^*_{\op{CR}, \op{amb}}(\cZ)$ as the composition 
\[\bar \tau_+:= \bar f_+ \circ \tau_+,\] where $\bar f_+: H^*_{\op{CR}}(\cY_+) \to H^*_{\op{CR}, \op{amb}}(\cZ)$ is the map from \eqref{e:barf}.  Let
\[ \mathbf{F}^{\cZ, \op{amb}} := H^*_{\op{CR}, \op{amb}}(\cZ) \otimes \cc O_{M^\circ_+ \times \CC_z}\]
and define
\[
\overline \nabla^{+, \op{amb}} := \bar \tau_+^*(\nabla^{\cZ, \op{amb}}).\]
\end{definition}

\begin{definition}
Define $\bar \tau_-: M^\circ_- \to \cc H_{\op{nar}}(w, G)$ as the composition \[\bar \tau_- := f \circ \tau_+,\] where $ f: H^*_{\op{CR}}(\cY_-) \to \cc H_{\op{nar}}(w, G)$ is the map from \eqref{e:f}.  Let
\[ \mathbf{F}^{w, G} := \cc H_{\op{nar}}(w, G) \otimes \cc O_{M^\circ_- \times \CC_z}\]
and define
\[
\overline \nabla^{-, \op{nar}} := \bar \tau_-^*(\nabla^{(w, G), \op{nar}}).\]
\end{definition}
\begin{definition}
Define 
\[
\overline \Theta_l^{\op{nar}} \in \hom \left( \cc H_{\op{nar}}(w, G), H^*_{\op{CR}, \op{amb}}(\cZ)\right) \otimes \cc O_{{M^\circ} \times \CC_z}\]
as $\overline \Theta_l^{\op{nar}} := \bar \Delta_+ \circ  \Theta_l^{\op{nar}} \circ \bar \Delta_-^{-1}$.
\end{definition}
Note that the factors of $2\pi i z$ in $\bar \Delta_+$ and $\bar \Delta_-^{-1}$ cancel, so 
$\overline \Theta_l^{\op{nar}}$ can also be written as $\Delta_+ \circ  \Theta_l^{\op{nar}}  \circ \Delta_-^{-1}$, which is still polynomial in $z$.

\begin{theorem}\label{t:LGCY} We have the following
\begin{itemize}
\item $\overline \nabla^{-, \op{nar}}$ and $\overline \nabla^{+, \op{amb}}$ are gauge equivalent via $ \overline \Theta_l^{\op{nar}}$, i.e.
\[ \overline \nabla^{+, \op{amb}} \circ \overline \Theta_l^{\op{nar}} =\overline \Theta_l^{\op{nar}} \circ \overline \nabla^{-, \op{nar}};\]
\item for all $E \in i^1_*(D(BG))$, 
\[ \overline \Theta_l^{\op{nar}} (\bs^{(w,G), nar}(E)(\bar \tau_-(t, \bt),z)) = \bs^{\cZ, \op{amb}}(\orlov(E))(\bar \tau_+(q, \bt), z);\]
\item $\overline \Theta_l^{\op{nar}}$ preserves the pairing up to a sign:
\[- S^{\cZ, \op{amb}}(\overline \Theta_l^{\op{nar}}(\alpha), \overline \Theta_l^{\op{nar}}(\beta)) = S^{(w, G), \op{nar}}(\alpha, \beta).\] 
\end{itemize}
\end{theorem}
\begin{proof}
For the first point, note the following chain of equalities:
\begin{align*}
\overline \Theta_l^{\op{nar}} & = \Delta_+ \circ \Theta_l^{\op{nar}} \circ \Delta_-^{-1} \\
& = \Delta_+ \circ L^{\cY_+, \op{nar}}( \tau_+, z) \circ \UU^{\op{nar}}_l \circ L^{\cY_-, \op{nar}}(\tau_-, z)^{-1} \circ \Delta_-^{-1} \\
& = L^{\cZ, \op{amb}}( \bar \tau_+, z) \circ e^{- \pi i dH/z} \circ \Delta_+ \circ  \UU^{\op{nar}}_l \circ \Delta_-^{-1} \circ L^{(w,G), nar}(\bar \tau_-, z)^{-1},
\end{align*}
where the first and second equalities are by definition and the third is Propositions~\ref{p:Lcom-} and~\ref{p:Lcom+}.  

For the second point we observe that for any $E \in D(\cY_-)_{BG}$,
\begin{align*}
& \overline \Theta_l^{\op{nar}} (\bs^{(w,G), nar}(i^1_* \circ \pi_*( E))(\bar \tau_-(t, \bt),z)) \\ =& 
  \overline \Theta_l^{\op{nar}} \circ \bar \Delta_- (\bs^{\cY_-, \op{nar}}(E)( \tau_-(t, \bt),z)) \\
  =& \bar \Delta_+ \circ  \Theta_l^{\op{nar}}  (\bs^{\cY_-, \op{nar}}(E)( \tau_-(t, \bt),z)) \\
  =& \bar \Delta_+  (\bs^{\cY_+, \op{nar}}(\vgit(E))( \tau_+(q, \bt),z)) \\
  =& \bs^{\cZ, \op{amb}}(j^* \circ \pi_* \circ \vgit(E))(\bar \tau_+(q, \bt), z) \\
 =& \bs^{\cZ, \op{amb}}(\orlov(i^1_* \circ \pi_*(E)))(\bar \tau_+(q, \bt), z).
\end{align*}
The first equality is Theorem~\ref{t:MLK}.  The second is by definition of $ \overline \Theta_l^{\op{nar}}$.  The third is Theorem~\ref{t:CTCnar}.  The fourth is Theorem~\ref{t:QSD}.  The fifth is due to the $K$-theoretic commuting diagram of Theorem~\ref{t:Ksquare}.

The third point follows from the second together with Proposition~\ref{p:pair}.  As we are in the Calabi--Yau setting, $e^{\pi i \sum q_j} = -1$.  Note that this sign discrepancy is consistent with the sign adjustments between the pairings in
Theorem~\ref{t:MLK}.

\end{proof}
\begin{remark}
A version of this theorem was given as Theorem~1.1 in \cite{CIR} in the case where $G$ is cyclic, i.e. $\PP(G) = \PP(c_1, \ldots, c_N)$ is a weighted projective space, where the correspondence between quantum $D$-modules was proven to be compatible with Orlov's equivalence.  For more general $G$, a version of the correspondence in terms of Lagrangian cones is Theorem~7.1 of \cite{LPS}, where we also made the connection with the crepant transformation conjecture.  The above theorem generalizes the first result and refines the second, showing that for general groups $G$, the LG/CY correspondence holds at the level of quantum $D$-modules, and is compatible with Orlov's equivalence.
\end{remark}

\begin{remark}
Tracing through the explicit formula for $\overline \Theta_l^{\op{nar}}$ as determined by $\overline \UU_l$, we see that the $l=0$ case agrees with  formula (7.1.1) in \cite{LPS}.  Furthermore, for all $l \in \ZZ$, this formula is consistent with \cite{CIR}.  More precisely, $e^{- \pi i dH/z} \circ \Delta_+ \circ  \UU^{\op{nar}}_l \circ \Delta_-^{-1}$ is equal to the non-equivariant limit of $\widetilde{\UU}^{\op{tw}}_l$ from Remark 4.13 of \cite{CIR} after multiplication by $-1$.  This sign discrepancy is explained in  Remark~\ref{r:shifty} above. 
\end{remark}

\subsection{The LG/CY cube}
We conclude with a pictorial description 
of the main results in this paper, represented below in the following commutative cube:
\begin{equation}
\begin{tikzcd}
 K(\cY_-)_{BG} \ar[rr, leftrightarrow, "\vgit"] \ar[rd, "i^1_* \circ \pi_*"] \ar[dd]& & K(\cY_+)_{\PP(G)}  \ar[rd, "j^* \circ \pi_*"] \ar[dd] &  \\
 & i^1_*\left(K(BG)\right) \ar[rr, leftrightarrow, crossing over, near start, "\orlov"] & &  j^*\left(K(\PP(G))\right) \ar[dd]\\
  QDM_{\op{nar}}(\cY_-) \ar[rr, leftrightarrow, near end, "\text{CTC}", , "\Theta_l^{\op{nar}}" ']  
  \ar[rd, leftrightarrow, "\text{local GW/FJRW}" ', "\bar \Delta_-"] & & QDM_{\op{nar}}(\cY_+)  \ar[rd, leftrightarrow, "\text{QSD}", "\bar \Delta_+" ']  &  \\
 & QDM_{\op{nar}}(w,G) \ar[from = uu, crossing over] \ar[rr, leftrightarrow, "\text{LG/CY}" ', , "\overline \Theta_l^{\op{nar}}"] 
  & &  QDM_{\op{amb}}(\cZ) 
\end{tikzcd}
\end{equation}
where for notational feasibility we have suppressed the changes of variables involved in each of the bottom arrows.  Here each of the vertical arrows is the map of integral structures $\bs^\square$. 

The commutativity of the top square is Theorem~\ref{t:Ksquare}.  The bottom left arrow and commutativity of the left square are the local GW/FJRW correspondence (Theorem~\ref{t:MLK}).  The bottom right arrow and commutativity of the right square are quantum Serre duality (Theorem~\ref{t:QSD}).  The bottom back arrow and commutativity of the back square are the narrow crepant transformation conjecture (Theorem~\ref{t:CTCnar}).  

From this diagram one immediately observes the existence of the bottom front arrow, obtained by composing the other three bottom arrows.  Commutativity of the bottom square then holds by construction, and commutativity of the front square follows from commutativity of all others.  We thus realize our philosophy that the (narrow) crepant transformation conjecture implies the LG/CY correspondence.  This is exactly the content of Section~\ref{ss:lgcyctc}.

\section{Appendix}

In this appendix we give a brief overview of 
``twisted'' Gromov--Witten and FJRW theory, where the virtual class is modified by a given characteristic class, and Givental's symplectic formalism.  While these tools are not always necessary for the statements of results in this paper, they are essential for the proofs of many of these results including Theorems~ \ref{t:MLK}, \ref{t:QSD}, and~\ref{t:mI}.  In particular, the  definitions and results of the appendix are used directly in Section~\ref{s:GWFJRW}.

\subsection{Twisted theories}\label{ss:twisty}
\subsubsection{FJRW theory}

Fix an LG pair $(w, G)$ as in Section~\ref{s:setup}.

Given parameters $s^j_k$ for $1 \leq j \leq N$ and $k \geq 0$, 
let $K$ denote the formal power series ring $\CC[[ s^j_k ]]_{1 \leq j \leq N, k \geq 0}$.
Define the formal invertible multiplicative characteristic class 
\[
\bs: \oplus \cc L_j \mapsto \exp \left( \sum_{j=1}^N \sum_{k \geq 0} s^j_k \ch_k (\cc L_j) \right) \in H^*_{CR}(\cX) \otimes_\CC K.\]

\begin{definition}
The \emph{twisted state space} $\cc H^{1, \bs}$ consists of $K$-linear combinations of elements $\phi_{g}$ indexed by $G$.
\[\cc H^{1, \bs} := \bigoplus_{g \in G} K \phi_{g}.\]
The \emph{narrow twisted state space} is defined as
\[\cc H^{1, \bs}_{\op{nar}} := \bigoplus_{g \in G_{\op{nar}}} K \phi_{gj^{-1}}.\]
\end{definition}

\begin{definition}\label{LG twisted inv}  
Given elements $g_i \in G$, integers $b_i \geq 0$ for $1 \leq i \leq n$, and an integer $h \geq 0$ with $2h-2 + n > 0$, define the genus $h$ $\bs$-twisted invariant of $(w, G)$ to be 
\[\langle \phi_{g_1j^{-1}} \psi^{b_1}, \ldots, \phi_{g_nj^{-1}} \psi^{b_n} \rangle^{1, \bs}_{h,n} :=
  \int_{[\sWbar_{h, \bg}(G)]}  \exp \left( \sum_{j=1}^N \sum_{k \geq 0} s^j_k \ch_k (\mathbb R \pi_* \cc L_j) \right) \left( \prod_{i=1}^n \psi_i^{b_i} \right).\]
  These invariants take values in $K$.
Invariants can be defined for general classes in $ \cc H^{1, \bs}$ by extending linearly.
\end{definition}

\begin{remark}
Note the shift in the above definition: insertions $\phi_{g_1 j^{-1}}, \ldots, \phi_{g_nj^{-1}}$ correspond to an integral on $\sWbar_{h, (g_1, \ldots, g_n)}(G)$, the moduli space of $W$-structures of Definition~\ref{d:wstr}.
\end{remark}

\begin{definition}\label{LG twisted pair}
The $\bs$-twisted pairing on $\cc H^{1, \bs}$ is defined by the relation
\begin{align*} \langle \phi_{g_1j^{-1}}, \phi_{g_2 j^{-1}} \rangle^{1, \bs} &:= \langle \phi_{g_1j^{-1}}, \phi_{g_2 j^{-1}}, \phi_{id} \rangle^{1, \bs}_{0, 3} \\
& =\exp\left( \sum_{j=1}^N \chi( \cc L_j)s^j_0\right)  \frac{\delta_{g_1g_2, id}}{|G|}.
\end{align*}
\end{definition}
%
%
%

The torus $\CC^*$ acts on the space of W-structures for $(w, G)$ by scaling each line bundle by a chosen factor.  Let $\lambda$ denote the equivariant parameter (the character of the standard representation of $\CC^*$) and let $-\lambda_j$ denote the character of the action on the $j$th line bundle $\cc L_j$.  We will assume that $\lambda_j$ is a nonzero multiple of $\lambda$.  
Let $e_{\CC^*}$ denote the equivariant Euler class.  After formally inverting $\lambda$, $e_{\CC^*}$ becomes invertible as well.  

Consider the specialization of the formal parameters 
given by  
\begin{equation}\label{e:spec}s_0^j = \ln( - 1/\lambda_j) \text{ and } s^j_k = (k-1)!/\lambda_j^k \text{ for } 1 \leq j \leq N, k >0.\end{equation}
Note then that $\bs( \oplus \cc L_j) = \tfrac{1}{e_{\CC^*}}( \oplus \cc L_j)$.  
With this specialization, the $\bs = e_{\CC^*}^{-1}$-twisted invariants of $(w, G)$ are equal to 
\[ \langle \phi_{g_1j^{-1}} \psi^{b_1}, \ldots, \phi_{g_nj^{-1}} \psi^{b_n} \rangle^{1, e_{\CC^*}^{-1}}_{h,n} = \int_{[\sWbar_{h, \bg}(G)]}    \frac{ \prod_{i=1}^n \psi_i^{b_i} }{e_{\CC^*}(\bigoplus \mathbb R \pi_* \cc L_j)}\]
and take values in $\CC[\lambda, \lambda^{-1}]$.  Furthermore, note that if $H^0(\cc C, \cc L_j ) = 0$ for all $j$ on all closed points of $\sWbar_{h, \bg}(G)$, then $\oplus \mathbb R^1 \pi_* ( \cc L_j)$ is a vector bundle.  With this specialization the invariant $\langle \phi_{g_1j^{-1}} \psi^{b_1}, \ldots, \phi_{g_nj^{-1}} \psi^{b_n} \rangle^{1, e_{\CC^*}^{-1}}_{h,n}$ lies in $\CC[\lambda]$, and has a well defined non-equivariant limit as $\lambda \mapsto 0$.

The pairing specializes to
\begin{equation}
 \langle \phi_{g_1j^{-1}}, \phi_{g_2 j^{-1}}, \phi_{id} \rangle^{1, e_{\CC^*}^{-1}}_{0, 3} = \prod_{k=1}^N \left(\frac{1}{-\lambda_k}\right)^{\lfloor 1 - m_k(g_1)\rfloor} \frac{\delta_{g_1g_2, id}}{|G|}.
\end{equation}
Let $\{\phi^{g}\}$ denote the dual basis of $\cc H^{1, \bs}$ with respect to the pairing.  Then 
\begin{equation}\label{e:dual}
\phi^{g j^{-1}} = |G| \prod_{k | m_k(g) = 0} (-\lambda_k) \phi_{g^{-1} j^{-1}}.
\end{equation}

\begin{lemma}\label{l:concave}  Assume $w$ is Fermat.
If the genus $h$ is equal to zero and all but at most one of $g_1, \ldots, g_n$ lie in $G_{\op{nar}}$, then $H^0(\cc C, \cc L_j) = 0$ for all $j$ on all closed points of $\sWbar_{h, \bg}(G)$.
\end{lemma}

\begin{proof}
Since $h = 0$ the dual graph of $\cc C$ is a tree.  

By our assumption that $w$ is Fermat, 
if $g_i$ is narrow then $m_j(g_i) \geq c_j/d$.  
If $\cc C$ is irreducible, then by \eqref{rig deg}, 
\[\deg(|\cc L_j|) = c_j/d(n-2) - \sum_{i=1}^n m_j(g_i) \leq c_j/d(n-2) - c_j/d(n-1) < 0.\]  The claim follows in this case since $H^0(\cc C, \cc L_j) = H^0(|\cc C|, |\cc L_j|)$.

More generally, note that by a similar calculation as above, the degree of $| \cc L_j|$ when restricted to an irreducible component is less than the number of nodes on that component.  Furthermore, given a section $s$ of $ \cc L_j$, for each component $\cc C '$ of $\cc C$ with one node and not containing $p_1$, $s$ must vanish identically on $\cc C'$, as follows from the $r= 1$ case.  Working in from these components, we see that $s$ vanishes on all of $\cc C$.
%
%
%
%
\end{proof}

\begin{lemma} \label{l:specFJRW}
If $g_1, \ldots, g_n$ all lie in $G_{\op{nar}}$, then 
\begin{align*} 
\lim_{\lambda \mapsto 0} \langle \phi_{g_1j^{-1}} \psi^{b_1}, \ldots, \phi_{g_nj^{-1}} \psi^{b_n} \rangle^{1, e_{\CC^*}^{-1}}_{0,n} =\langle \varphi_{g_1j^{-1}} \psi^{b_1}, \ldots, \varphi_{g_nj^{-1}} \psi^{b_n} \rangle^{(w, G)}_{0,n} .
\end{align*}
If  $g_1, \ldots, g_{n-1}$ lie in $G_{\op{nar}}$, and $\theta_{g_n}$ lies in $\cc H_{g_n}(W,G)$ where $g_n$ is broad, then 
\[\lim_{\lambda \mapsto 0} \langle \phi_{g_1j^{-1}} \psi^{b_1}, \ldots, \phi_{g_nj^{-1}} \psi^{b_n} \rangle^{1, e_{\CC^*}^{-1}}_{0,n}  \phi^{g_nj^{-1}} = 0.\]
\end{lemma}

\begin{proof}
The first statement is part (5a) of Theorem~4.1.8 of \cite{FJR1}. 
By Lemma \ref{l:concave}, the limit $\lim_{\lambda \mapsto 0} \langle \phi_{g_1j^{-1}} \psi^{b_1}, \ldots, \phi_{g_nj^{-1}} \psi^{b_n} \rangle^{1, e_{\CC^*}^{-1}}_{0,n}$ exists.  For $g_n$ broad, the dual element $\phi^{g_nj^{-1}}$ contains a positive power of $\lambda$ by \eqref{e:dual}.
We conclude that
\[\lim_{\lambda \mapsto 0} \langle \phi_{g_1j^{-1}} \psi^{b_1}, \ldots, \phi_{g_nj^{-1}} \psi^{b_n} \rangle^{1, e_{\CC^*}^{-1}}_{0,n}  \phi^{g_nj^{-1}} = 0.\]

\end{proof}
\subsubsection{GW theory}\label{ss:twisty2}
Let $\cV_1, \ldots, \cV_N$ be a collection of line bundles on a smooth and proper Deligne--Mumford stack $\cX$.  
As in the case of FJRW theory, one can define twisted GW invariants depending on parameters $s_k^j$ defining a multiplicative characteristic class and taking values in $K$.  

\begin{definition}
Given $\alpha_1, \ldots, \alpha_n \in H^*_{\op{CR}}(\cX)$ and integers $b_1, \ldots, b_n \geq 0$ define the \emph{$\bs$-twisted}  Gromov--Witten invariant of $\cX$ to be
\[ \langle \alpha_1 \psi^{b_1}, \ldots, \alpha_n \psi^{b_n} \rangle^{\cX, \bs}_{h, n} := 
\int_{[\sMbar_{h,n}(\cc X, d)]^{vir}} \exp \left( \sum_{j=1}^N \sum_{k \geq 0} s^j_k \ch_k (\mathbb R \pi_* f^* \cV_j) \right)  \prod_{i=1}^n ev_i^*( \alpha_i) \cup \psi_i^{b_i}.\]
\end{definition}

\begin{definition}
The \emph{$\bs$-twisted GW state space} of $\cX$ is $H^*_{\op{CR}}(\cX)\otimes K$.  We define the \emph{$\bs$-twisted  pairing} to be
\[ \langle \alpha, \beta \rangle^{ \cX, \bs} := \left(\exp \left( \sum_{j=1}^N \sum_{k \geq 0} s^j_k \ch_k ( \cV_j) \right) \alpha, I_*(\beta) \right).\]
\end{definition}

\begin{notation}
Consider the particular case where $\cX = BG$ and $\cV_j = \CC_j$ is the $j$th factor in $\cY_-$.  Let $T = \CC^*$ act on $\CC_j$ with weight $-\lambda_j$.  In this case, we denote $\langle -, \ldots, - \rangle^{BG, \bs}_{h, n}$ by $\langle -, \ldots, - \rangle^{0, \bs}_{h, n}.$
\end{notation}

Exactly as before, we invert the equivariant parameter and specialize the variables 
\begin{equation}\label{e:twistedGW} s_0^j = \ln( - 1/\lambda_j) \text{ and } s^j_k = (k-1)!/\lambda_j^k \text{ for } 1 \leq j \leq N, k >0\end{equation}
so that $\bs( \oplus \CC_j) = \tfrac{1}{e_{\CC^*}}( \oplus \CC_j)$.  

\begin{remark}\label{r:twistedGW}
In this case, the specialized $\bs$-twisted GW invariants of $BG$ are equal to the $T$-equivariant GW invariants of $
\cY_-$.  When the non-equivariant limit exists, these specialize to the usual GW invariants of $\cY_-$.
\end{remark}

\begin{lemma} \label{l:reltotwisted}
Under the specialization of $\bs$ as in \eqref{e:twistedGW},
$\langle \ii_g, \bt \rangle^{0, \bs}  $ specializes to the equivariant Gromov--Witten invariant $\langle \ii_g, \bt \rangle^{\cY_-^T} $.  When $g$ is narrow, there is a well-defined non-equivariant limit.  When $h$ is broad, $\langle \ii_g, \bt, \ii_h\rangle^{\cY_-} \ii^h $ is zero in the non-equivariant limit.  
%
\end{lemma}

\begin{proof}
Under the specialization of $\bs$ as in \eqref{e:twistedGW},  $(0, \bs)$ invariants become (equivariant) Gromov--Witten invariants of $\cY_-$ by \cite{GP}. 
When $g$ is narrow, the limit  $\lim_{\lambda \mapsto 0} \langle \ii_g, \bt, \ii_h\rangle^{\cY_-}$ exists.  If $h$ is broad, then $\ii^h$ contains a positive power of $\lambda$.  The second statement follows.
\end{proof}

\subsection{Givental formalism}
We give here an extremely brief review of Givental's symplectic formalism, recalling only the facts necessary for the proof of Proposition~\ref{p:Lcom-}.
For a detailed exposition  see \cite{G3}.

Let $\square$ denote either FJRW theory $(w, G)$, GW theory on a compact orbifold $\cX$, or a twisted theory (such as $(0, \bs)$ or $(1, \bs)$).  Denote by $\sV^\square$ the space 
\[ \sV^\square := H^\square ((z^{-1}))[[\bs]]\]
where $\bs = 0$ in the untwisted case.  This vector space is endowed with a symplectic pairing given by
\[\Omega^\square(f_1(z), f_2(z)) := \op{Res}_{z=0}\br{f_1(-z), f_2(z) }^\square\]
where $\br{ - , -}^\square$ denotes the state space pairing for $H^\square$.  There is a natural polarization
\[\sV^\square = \sV^\square_+ \oplus \sV^\square_- = H^\square [z][[\bs]] \oplus H^\square [[z^{-1}]][[\bs]].\]
Given a basis $\{T_i\}_{i \in I}$ for $H^\square$, we obtain Darboux coordinates $\{q^i_k, p_{k,i}\}_{i \in I, k \geq 0}$ with respect to the polarization.  A general point of $\sV^\square$ may be written as 
\[
\sum_{k \geq 0}\sum_{i \in I} q_k^i T_i z^k + \sum_{k \geq 0}\sum_{i \in I} p_{k,i}T^i (-z)^{-k-1}.
\]
\begin{definition}\label{cone}
The \emph{overruled Lagrangian cone} (sometimes called simply the Lagrangian cone)  $\sL^\square$ is the subspace of $\sV^\square$ parametrized as
\begin{equation}\label{e:cone2}
- z+ \sum_{\substack{k \geq 0 \\ i \in I}} t_k^i T_i z^k +\sum_{\substack{a_1, \ldots , a_n,  a\geq 0 \\ i_1, \ldots , i_n, i \in I}} \frac{t^{i_1}_{a_1}\cdots t^{i_n}_{a_n}}{n!(-z)^{a+1}}\langle \psi^aT_i ,\psi^{a_1}T_{i_1},\dots,\psi^{a_n}T_{i_n}\rangle_{0, n+1}^\square T^i,
\end{equation}
where the parameters varying are the $\{t^i_k\}_{k \geq 0, i \in I}$.
\end{definition}
This space can also be defined as a shift of the differential of the genus zero descendant potential function which shows that it is a Lagrangian subspace.  Due to various universal identities in GW and FJRW theory, $\sL^\square$ satisfies the following 
\begin{enumerate}
\item it is a cone;
\item it is ``overruled:'' for all $f \in \sL^\square$,
\[\sL^\square \cap \cc T_f\sL = z \cc T_f \sL,\]
where $\cc T_f\sL$ denotes the tangent space at $f$.
\end{enumerate}
Givental's $J$-function is given by
\[
J^\square(\bt, z) := z + \bt + \sum_{n \geq 0} \sum_{ i \in I} \frac{1}{n!} \br{ \frac{T_i}{z - \psi}, \bt, \ldots, \bt }^\square_{0, n+1} T^i
\]  
for $\bt \in H^\square$.
Note that $J^\square(\bt, -z) = -z \oplus \bt \oplus \sV^- \cap \sL^\square$.  Due to the aforementioned properties of $\sL^\square$, the $J$-function fully determines the Lagrangian cone.  More precisely, the cone is a union of $z$ times the tangent spaces at $J^\square(\bt, -z)$,
\begin{equation}\label{e:Jg0}
\sL^\square = \set{z\cc T_{J^\square(\bt, -z)} \sL | \bt \in H^\square},
\end{equation}
and the tangent spaces take a particularly simple form,
\begin{equation}\label{e:Jg}
z\cc T_{J^\square(\bt, -z)}\sL = \set{J^\square(\bt, -z) + z\sum c_i(z) \frac{\partial}{\partial t^i} J^\square(\bt, -z) | c_i(z) \in \CC[z]}.
\end{equation}
It follows from the string equation (Section 26.3 of \cite{HV}) that if we assume $T_0$ is the identity in $\cc H^\square$, then $z \frac{\partial}{\partial t^0} J^\square(\bt, z) = J^\square(\bt, -z)$ and
\begin{equation}\label{e:Jtan}
\cc T_{J^\square(\bt, -z)}\sL = \set{ \sum c_i(z) \frac{\partial}{\partial t^i} J^\square(\bt, -z) | c_i(z) \in \CC[z]}.
\end{equation}

\bibliographystyle{plain}
\bibliography{references}

\begin{thebibliography}{10}

\bibitem{AGV}
Dan Abramovich, Tom Graber, and Angelo Vistoli.
\newblock Gromov-{W}itten theory of {D}eligne-{M}umford stacks.
\newblock {\em Amer. J. Math.}, 130(5):1337--1398, 2008.

\bibitem{AV}
Dan Abramovich and Angelo Vistoli.
\newblock Compactifying the space of stable maps.
\newblock {\em J. Amer. Math. Soc.}, 15(1):27--75, 2002.

\bibitem{BFK2}
Matthew Ballard, David Favero, and Ludmil Katzarkov.
\newblock A category of kernels for equivariant factorizations and its
  implications for {H}odge theory.
\newblock {\em Publ. Math. Inst. Hautes \'{E}tudes Sci.}, 120:1--111, 2014.

\bibitem{BFK}
Matthew Ballard, David Favero, and Ludmil Katzarkov.
\newblock Variation of geometric invariant theory quotients and derived
  categories.
\newblock {\em J. Reine Angew. Math.}, 746:235--303, 2019.

\bibitem{BF}
K.~Behrend and B.~Fantechi.
\newblock The intrinsic normal cone.
\newblock {\em Invent. Math.}, 128(1):45--88, 1997.

\bibitem{BH}
Lev~A. Borisov and R.~Paul Horja.
\newblock On the {$K$}-theory of smooth toric {DM} stacks.
\newblock In {\em Snowbird lectures on string geometry}, volume 401 of {\em
  Contemp. Math.}, pages 21--42. Amer. Math. Soc., Providence, RI, 2006.

\bibitem{ChenR1}
Weimin Chen and Yongbin Ruan.
\newblock A new cohomology theory of orbifold.
\newblock {\em Comm. Math. Phys.}, 248(1):1--31, 2004.

\bibitem{CIR}
Alessandro Chiodo, Hiroshi Iritani, and Yongbin Ruan.
\newblock Landau-{G}inzburg/{C}alabi-{Y}au correspondence, global mirror
  symmetry and {O}rlov equivalence.
\newblock {\em Publ. Math. Inst. Hautes \'Etudes Sci.}, 119:127--216, 2014.

\bibitem{ChR}
Alessandro Chiodo and Yongbin Ruan.
\newblock Landau-{G}inzburg/{C}alabi-{Y}au correspondence for quintic
  three-folds via symplectic transformations.
\newblock {\em Invent. Math.}, 182(1):117--165, 2010.

\bibitem{CFFGKS}
Ionut Ciocan-Fontanine, David Favero, Bumsig Kim, J\'er\'emy Gu\'er\'e, and
  Mark Shoemaker.
\newblock Fundamental factorization of a {GLSM}, part {I}: Construction.
\newblock arXiv:1802.05247, 2018.

\bibitem{ClRo}
Emily Clader and Dustin Ross.
\newblock Sigma models and phase transitions for complete intersections.
\newblock {\em Int. Math. Res. Not.}, (15):4799--4851, 2018.

\bibitem{ClRu}
Emily Clader and Yongbin Ruan.
\newblock {\em B-Model {G}romov-{W}itten Theory}.
\newblock Trends in mathematics. Birkh\"auser, Cham, Switzerland, 2018.

\bibitem{CCIT}
Tom Coates, Alessio Corti, Hiroshi Iritani, and Hsian-Hua Tseng.
\newblock Computing genus-zero twisted {G}romov-{W}itten invariants.
\newblock {\em Duke Math. J.}, 147(3):377--438, 2009.

\bibitem{CCIT2}
Tom Coates, Alessio Corti, Hiroshi Iritani, and Hsian-Hua Tseng.
\newblock A mirror theorem for toric stacks.
\newblock {\em Compos. Math.}, 151(10):1878--1912, 2015.

\bibitem{CGIJJM}
Tom Coates, Amin Gholampour, Hiroshi Iritani, Yunfeng Jiang, Paul Johnson, and
  Cristina Manolache.
\newblock The quantum {L}efschetz hyperplane principle can fail for positive
  orbifold hypersurfaces.
\newblock {\em Math. Res. Lett.}, 19(5):997--1005, 2012.

\bibitem{CG}
Tom Coates and Alexander Givental.
\newblock Quantum {R}iemann-{R}och, {L}efschetz and {S}erre.
\newblock {\em Ann. of Math. (2)}, 165(1):15--53, 2007.

\bibitem{CIJ}
Tom Coates, Hiroshi Iritani, and Yunfeng Jiang.
\newblock The crepant transformation conjecture for toric complete
  intersections.
\newblock {\em Adv. Math.}, 329:1002--1087, 2018.

\bibitem{CIJS}
Tom Coates, Hiroshi Iritani, Yunfeng Jiang, and Ed~Segal.
\newblock {$K$}-theoretic and categorical properties of toric
  {D}eligne-{M}umford stacks.
\newblock {\em Pure Appl. Math. Q.}, 11(2):239--266, 2015.

\bibitem{CIT}
Tom Coates, Hiroshi Iritani, and Hsian-Hua Tseng.
\newblock Wall-crossings in toric {G}romov-{W}itten theory. {I}. {C}repant
  examples.
\newblock {\em Geom. Topol.}, 13(5):2675--2744, 2009.

\bibitem{CR}
Tom Coates and Yongbin Ruan.
\newblock Quantum cohomology and crepant resolutions: a conjecture.
\newblock {\em Ann. Inst. Fourier (Grenoble)}, 63(2):431--478, 2013.

\bibitem{FJR1}
Huijun Fan, Tyler Jarvis, and Yongbin Ruan.
\newblock The {W}itten equation, mirror symmetry, and quantum singularity
  theory.
\newblock {\em Ann. of Math. (2)}, 178(1):1--106, 2013.

\bibitem{FJR15}
Huijun Fan, Tyler Jarvis, and Yongbin Ruan.
\newblock A mathematical theory of the gauged linear sigma model.
\newblock {\em Geom. Topol.}, 22(1):235--303, 2017.

\bibitem{G1}
Alexander~B. Givental.
\newblock Equivariant {G}romov-{W}itten invariants.
\newblock {\em Int. Math. Res. Not.}, (13):613--663, 1996.

\bibitem{G3}
Alexander~B. Givental.
\newblock Symplectic geometry of {F}robenius structures.
\newblock In {\em Frobenius manifolds}, Aspects Math., E36, pages 91--112.
  Friedr. Vieweg, Wiesbaden, 2004.

\bibitem{GP}
T.~Graber and R.~Pandharipande.
\newblock Localization of virtual classes.
\newblock {\em Invent. Math.}, 135(2):487--518, 1999.

\bibitem{Hirano}
Yuki Hirano.
\newblock Derived {K}n{\"o}rrer periodicity and {O}rlov's theorem for gauged
  {L}andau--{G}inzburg models.
\newblock {\em Compositio Mathematica}, 153(5):973--1007, 2017.

\bibitem{HV}
Kentaro Hori, Sheldon Katz, Albrecht Klemm, Rahul Pandharipande, Richard
  Thomas, Cumrun Vafa, Ravi Vakil, and Eric Zaslow.
\newblock {\em Mirror symmetry}, volume~1 of {\em Clay Mathematics Monographs}.
\newblock American Mathematical Society, Providence, RI; Clay Mathematics
  Institute, Cambridge, MA, 2003.
\newblock With a preface by Vafa.

\bibitem{Iri}
Hiroshi Iritani.
\newblock An integral structure in quantum cohomology and mirror symmetry for
  toric orbifolds.
\newblock {\em Adv. Math.}, 222(3):1016--1079, 2009.

\bibitem{Iri2}
Hiroshi Iritani.
\newblock Ruan's conjecture and integral structures in quantum cohomology.
\newblock In {\em New developments in algebraic geometry, integrable systems
  and mirror symmetry ({RIMS}, {K}yoto, 2008)}, volume~59 of {\em Adv. Stud.
  Pure Math.}, pages 111--166. Math. Soc. Japan, Tokyo, 2010.

\bibitem{Iri3}
Hiroshi Iritani.
\newblock Quantum cohomology and periods.
\newblock {\em Ann. Inst. Fourier (Grenoble)}, 61(7):2909--2958, 2011.

\bibitem{IMM}
Hiroshi Iritani, Etienne Mann, and Thierry Mignon.
\newblock Quantum {S}erre theorem as a duality between quantum {$D$}-modules.
\newblock {\em Int. Math. Res. Not.}, (9):2828--2888, 2016.

\bibitem{Isik}
Mehmet~Umut Isik.
\newblock Equivalence of the derived category of a variety with a singularity
  category.
\newblock {\em Int. Math. Res. Not.}, (12):2787--2808, 2013.

\bibitem{Jia}
Yunfeng Jiang.
\newblock The orbifold cohomology ring of simplicial toric stack bundles.
\newblock {\em Illinois J. Math.}, 52(2):493--514, 2008.

\bibitem{LPS}
Yuan-Pin Lee, Nathan Priddis, and Mark Shoemaker.
\newblock A proof of the {L}andau-{G}inzburg/{C}alabi-{Y}au correspondence via
  the crepant transformation conjecture.
\newblock {\em Ann. Sci. \'Ec. Norm. Sup\'er. (4)}, 49(6):1403--1443, 2016.

\bibitem{PV3}
Alexander Polishchuk and Arkady Vaintrob.
\newblock Matrix factorizations and singularity categories for stacks.
\newblock {\em Ann. Inst. Fourier (Grenoble)}, 61(7):2609--2642, 2011.

\bibitem{PV2}
Alexander Polishchuk and Arkady Vaintrob.
\newblock Chern characters and {H}irzebruch-{R}iemann-{R}och formula for matrix
  factorizations.
\newblock {\em Duke Math. J.}, 161(10):1863--1926, 2012.

\bibitem{PV}
Alexander Polishchuk and Arkady Vaintrob.
\newblock Matrix factorizations and cohomological field theories.
\newblock {\em J. Reine Angew. Math.}, 714:1--122, 2016.

\bibitem{Seg}
Ed~Segal.
\newblock Equivalence between {GIT} quotients of {L}andau-{G}inzburg
  {B}-models.
\newblock {\em Comm. Math. Phys.}, 304(2):411--432, 2011.

\bibitem{Shi}
Ian Shipman.
\newblock A geometric approach to {O}rlov's theorem.
\newblock {\em Compos. Math.}, 148(5):1365--1389, 2012.

\bibitem{Sh1}
Mark Shoemaker.
\newblock Narrow quantum {D}-modules and quantum {S}erre duality.
\newblock arXiv:1811.01888, 2018.

\end{thebibliography}

\end{document}